\documentclass[10pt,a4paper,english]{article}

\usepackage[english]{babel}
\usepackage[utf8]{inputenc}
\usepackage[T1]{fontenc}
\usepackage{
amsmath,
amsthm,
amsfonts,
amssymb,
colortbl,
dsfont,
mathrsfs,
fullpage,
placeins,
xcolor
}
\usepackage[pdftex]{hyperref} 
\hypersetup{
 colorlinks=true,
 linkcolor=gray, 
 citecolor=gray, 
 urlcolor=gray } 

\usepackage{graphicx} 

\newtheorem*{theorem*}{Theorem}
\newtheorem{theorem}{Theorem}[section]
\newtheorem{corollary}[theorem]{Corollary}
\newtheorem{conjecture}{Conjecture}[section]
\newtheorem{proposition}[theorem]{Proposition}
\newtheorem{lemma}[theorem]{Lemma}
\theoremstyle{definition}
\newtheorem{definition}{Definition}[section]
\newtheorem{remark}{Remark}[section]

\newcommand{\Alpha}{\vector{\boldsymbol \alpha}}
\newcommand{\alphap}{\alpha_p}

\newcommand{\Alphaq}{\Alpha_q}
\newcommand{\alphaq}{\alpha_q}

\newcommand{\calPb}{\vector{\pmb{\mathscr{P}}}}

\newcommand{\dd}{{\rm ~d}}
\newcommand{\ddt}{\dfrac{\rm d}{{\rm d}\,t}}
\renewcommand{\div}{{\rm ~div}}
\newcommand{\divs}{{\mbox{\rm \tiny div}}}
\newcommand{\dsp}{\displaystyle}

\newcommand{\eqdef}{:=}
\newcommand{\eg}{\textit{e.g.} }
\newcommand{\equ}{\underline{e_q}}
\newcommand{\epu}{\underline{e_p}}

\newcommand{\Forall}{ \quad \forall }

\newcommand{\grad}{ ~\vector{grad}}
\definecolor{darkgreen}{rgb}{0.55, 0.71, 0.0}

\renewcommand{\H}{\mathbf{H}}
\newcommand{\Ham}{\mathcal{H}}

\newcommand{\ie}{\textit{i.e.} }
\newcommand{\Ib}{\vector{I}}

\renewcommand{\L}{\mathbf{L}}

\newcommand{\matl}{\left(\begin{matrix}}
\newcommand{\matr}{\end{matrix}\right)}

\newcommand{\n}{\vector{n}}
\newcommand{\N}{\mathbb{N}}
\newcommand{\nol}{\left\Vert}
\newcommand{\nor}{\right\Vert}

\renewcommand{\P}{\mathbb{P}}
\newcommand{\Pb}{\vector{P}}
\newcommand{\phiq}{\vector{\boldsymbol \varphi}_q}
\newcommand{\Phiq}{\vector{\pmb{\Phi}}_q}
\newcommand{\phip}{\varphi_p}
\newcommand{\Phip}{\Phi_p}

\newcommand{\psl}{\left\langle}
\newcommand{\psr}{\right\rangle}

\newcommand{\R}{\mathbb{R}}
\newcommand{\red}[1]{\textcolor{red}{#1}}

\newcommand{\s}{\vector{s}}

\renewcommand{\span}{{\mbox{\rm Span}}}

\newcommand{\Tau}{\mathcal{T}}
\newcommand{\Tens}{\overline{\overline{\boldsymbol T}}}

\renewcommand{\v}{\vector{v}}
\newcommand{\V}{\vector{V}}
\newcommand{\val}{\left\vert}
\newcommand{\var}{\right\vert}
\renewcommand{\vector}[1]{\overrightarrow{\bf #1}}
\newcommand{\via}{\textit{via} }

\newcommand{\x}{\vector{x}}

\newcommand{\y}{\vector{y}}

\newcommand{\z}{\vector{z}}

\title{\textbf{Numerical analysis of a structure-preserving space-discretization for an anisotropic and heterogeneous boundary controlled \texorpdfstring{$N$}{N}-dimensional wave equation as port-Hamiltonian system}}

\author{\textsc{Haine}, Ghislain \and \textsc{Matignon}, Denis \and \textsc{Serhani}, Anass}

\date{}

\begin{document}

\maketitle

\begin{abstract}
{The anisotropic and heterogeneous \texorpdfstring{$N$}{N}-dimensional wave equation, controlled and observed at the boundary, is considered as a port-Hamiltonian system. A recent structure-preserving mixed Galerkin method is applied, leading directly to a finite-dimensional port-Hamiltonian system: its numerical analysis is carried out in a general framework. Compatibility conditions are then exhibited to reach the best trade-off between the convergence rate and the number of degrees of freedom for both the state error and the Hamiltonian error. In particular, the order of boundary approximations is discussed, keeping the port-Hamiltonian formalism in mind. Numerical simulations in 2D are performed to illustrate the optimality of the main theorems among several choices of classical finite element families.}


\end{abstract}

\section{Introduction}

The present work addresses the numerical analysis of a structure-preserving space-discretization of an $N$-dimensional wave equation with boundary control in the formalism of port-Hamiltonian systems. Since it is intended to merge several points of view on the same subject, the authors have taken care to be pedagogical in each section, hence trying to talk to several scientific communities. This choice of presentation will certainly lead readers to find some parts related to his/her domain(s) of research unnecessary. Roughly speaking, this paper is intended to port-Hamiltonian specialists, numerical analysts, and scientific computing users.

\subsection{Port-Hamiltonian systems}

In the last two decades, infinite-dimensional port-Hamiltonian systems (pHs)~\cite{SchMas02,Rashad2020} have proved to be a very accurate way to model and control complex multi-physics open systems. This framework enjoys several advantages, such as a relevant physical meaning and a useful underlying geometrical structure (namely Stokes-Dirac structure). It has to be pointed out that, even if known Partial Differential Equations (PDEs) are often \emph{only rewritten} in the pHs formalism in general, this powerful tool also allows a direct modelling of physical systems (see for instance~\cite{DuiMacStrBru09,SerMatHai19b,altmann2020porthamiltonian}) which proves useful to derive PDEs. Furthermore, it is intrinsically modular: interaction systems (such as fluid-structure interactions \cite{CarMatPom20}, heat-wave interactions~\cite{HaiMatMon2022}, plasma in a tokamak \cite{VuLefMas16}, etc.) can be described through the interconnection of several subsystems with a port-Hamiltonian structure, leading to a more complex pHs \cite{CerSchBan07,Kurula2010}. It finally leads to a power balance, expressing the variation of the Hamiltonian functional (often chosen as the system total energy), especially \via boundary controls and boundary observations.

\subsection{Structure-preserving discretisation}

A recent topic of research is to provide accurate (space-)~discretization methods to preserve this powerful formalism. Roughly speaking, mainly two non-exclusive communities work on the issue of {\em structure-preserving discretization}. The first one makes use of exterior calculus, while the other makes use of vector calculus. It is known that the two points of view are well-founded, and several strategies to merge their advantages efficiently have already been proposed for several discretization issues (see \eg \cite{KotHDR} and the many references therein).

In the present work, a method for the preservation of the power balance of the Hamiltonian (encoded in an underlying Stokes-Dirac structure) is studied. In the wide literature, several strategies have been proposed: we can cite \eg \cite{Hiemstra2014,SesSchSch14,KotMasLef18} for geometric discretizations,~\cite{Lee2018b,Egger2019} for Galerkin methods and~\cite{TreRamGorKot18} for finite differences method. However, some of these strategies seem difficult to carry over to $N$-dimensional systems or to apply to complex geometries, while others require post-processing to construct the finite-dimensional Dirac structure. Another structure-preserving community works on the preservation of the de Rham cohomology and related decompositions (such as the Hodge-Helmholtz decomposition). This topic is older and finds its origins in problems such as electromagnetism (see \eg \cite{Mon03} and references therein). It is often written in the exterior calculus formalism, allowing for more abstraction, hence more generality, for the construction of discrete differential operators: see \eg \cite{Hiptmair2001b,BocHym06,Arnold2006,ArnFalWin10,Gerritsma2016} for theoretical aspects, and~\cite{FarKliJocFloDyc13,FarBalDyc14,Gerritsma2018} for some applications to partial differential equations.

According to these definitions of structure-preserving discretization, a numerical method for port-Ha\-mil\-to\-nian systems should be able to take into account the aforementioned continuous properties at the discrete level. Indeed, this would lead to a relevant physical meaning for the computed quantities (without post-processing), together with an obvious manner to distribute the computations thanks to the modularity property: in particular, each sub-system could be reduced through a structure-preserving model reduction~\cite{GugPolBeaSch12,EggKugLilMarMeh18,HaiMat21} \emph{prior to} their interconnections. Furthermore, in the field of automatic control, several methodologies for efficient control or stabilisation rely on the pHs form of the approximate finite-dimensional system~\cite{Toledo2020}: this encourages research for efficient structure-preserving methods of infinite-dimensional pHs, and even more  those related to boundary-controlled-and-observed PDEs.

A special case of the mixed Galerkin method, called the Partitioned Finite Element Method (PFEM)~\cite{CarMatLef19}, seems to be one of the most adapted scheme to build a mimetic finite-dimensional Dirac structure~\cite{Rashad2020}.

In~\cite{Jol03}, the numerical method proposed for the spatial discretization of \emph{closed} hyperbolic systems, based on the \emph{primal-dual} or \emph{dual-primal formulations} given in \cite[Eqs.~(15) and~(16)]{Jol03}, and making use of an abstract mixed Galerkin method, can be seen as the starting point of the PFEM for \emph{closed} systems. Indeed, the idea of \emph{partitioning} the system to choose on which equation an integration by parts  should be applied was already mentioned: ``the principle is to multiply the two equations [...] by test functions and to integrate over $\Omega$, but the key point this time is to apply integration by parts only for one of the two equations.'', see \cite[p.~207]{Jol03}. The new difficulty for port-Hamiltonian systems lies in the boundary terms, namely the control and the observation. The present work investigates the issue of accurate approximations at the boundary.

Indeed, the numerical analysis of boundary controlled wave-like systems discretized \via the Mixed Finite Element Method often makes use of  known results on elliptic systems. This is an easy way to obtain convergence rates, but it definitely strengthens the conditions on the finite element families that can be used, for instance by introducing an artificial need of so-called \emph{inf--sup condition}. In the case of Dirichlet boundary control, this makes the numerical schemes artificially complicated, requiring the discretization of a lifting operator (\ie solving an elliptic system at each time step) or the addition of Lagrange multipliers. To the best of our knowledge, numerical analysis without these difficulties has only been performed on particular choices of finite elements for the former case, see \eg \cite[Remark~6]{Jol03} and for instance~\cite{BecJolTso00,BecJolTso02}, where new families of mixed finite elements are constructed on purpose. The present work extends this result to {\em open} dynamical systems for all conforming finite elements and both Neumann and Dirichlet boundary controls, without such an inf--sup condition, neither the discretization of a lifting operator, nor the use of Lagrange multipliers. The less restrictive conforming conditions (for closed systems) have already been stated in~\cite[Eq.~(31)]{Jol03}, and claimed to be required in~\cite[Section~7.9]{BofBreFor13} for efficient convergence (with usual finite elements); in~\cite{LepMorRod14} a similar result has already been obtained for the Timoshenko beam in 1-D. Furthermore, it is shown that adding compatibility conditions between the finite element families in order to preserve de Rham cohomology results in a better convergence rate for the discrete Hamiltonian towards the continuous one, see Theorem~\ref{th:General-Hamiltonian}.

\subsection{Statement of the main results}

The objective of this section is to provide an informal statement of the main result. In this work, the aim is to analyse the convergence of the PFEM, applied on the following system, associated to the $N$-dimensional anisotropic and heterogeneous wave equation
\begin{equation}\label{eq:waves-continuous}
\left\{\begin{array}{ll}
\rho(\x) \; \partial^2_t w(t,\x) - \div \left( \Tens(\x) \; \grad (w(t,\x)) \right) = 0,& \Forall \x \in \Omega, t \ge 0,\\
w(0,\x) = w_0(\x), \qquad \partial_t w(0,\x) = w_1(\x),& \Forall \x \in \Omega,
\end{array}\right.
\end{equation}
together with the following collocated boundary control $u$ and boundary observation $y$
\begin{equation}\label{eq:waves-boundary}
\left\{\begin{array}{ll}
u(t,\x) = \left( \Tens(\x) \; \grad (w(t,\x)) \right)^\top \; \n(\x),& \Forall \x \in \partial\Omega, t \ge 0,\\
y(t,\x) = \partial_t w(t,\x),& \Forall \x \in \partial\Omega, t \ge 0.
\end{array}\right.
\end{equation}
In these equations
\begin{itemize}
\item
$\Omega$ is an open bounded domain of $\R^N$, $N = 1, 2, 3$, with Lipschitz boundary $\partial\Omega$;
\item
$\n$ is the outward normal at the boundary $\partial\Omega$;
\item
$w(t,\x)$ is the deflection from the equilibrium position at point $\x \in \Omega$ and time $t \ge 0$;
\item
$u$ is the boundary control corresponding to forces applied at the boundary;
\item
$y$ is the collocated boundary observation corresponding to the measured velocities on $\partial\Omega$;
\item
$\rho$ is the mass density, supposed to be bounded from above and below (almost everywhere) by $\rho^+$ and $\rho_->0$ respectively;
\item
$\Tens$ is Young's elasticity modulus, supposed to be a real symmetric tensor bounded from above and below (almost everywhere, in the matrix-norm sense) by $T^+ \overline{\overline{\boldsymbol I}}$ and $T_- \overline{\overline{\boldsymbol I}}$ respectively, where $T_->0$ and $\overline{\overline{\boldsymbol I}}$ is the identity tensor;
\item
$\top$ stands for the transpose of vectors or matrices.
\end{itemize}
We associate to system~\eqref{eq:waves-continuous}--\eqref{eq:waves-boundary} the \textit{Hamiltonian}
$$
\Ham(t) \eqdef \dfrac{1}{2} \int_\Omega \left[ \left( \grad \left(w(t,\x)\right) \right)^\top \; \Tens(\x) \; \grad \left(w(t,\x)\right) + \rho(\x) \; \left(\partial_t w(t,\x)\right)^2\right] \dd \x,
$$
made of the potential and kinetic energies of the physical system.

The following statement is an abridged and formal version of Theorems~\ref{th:General-State} and~\ref{th:General-Hamiltonian}.
\begin{theorem*}
Let us denote the \emph{strain} $\Alphaq \eqdef \grad (w)$, the \emph{linear momentum} $\alphap \eqdef \rho \partial_t w$, and their discrete counterparts $\Alphaq^d$ and $\alphap^d$ obtained by the Partitioned Finite Element Method. The Hamiltonian then rewrites
\begin{equation}\label{eq:Hamiltonian}
\Ham(t) \eqdef \Ham(\Alphaq(t),\alphap(t)) \eqdef \dfrac{1}{2} \int_\Omega \left[ \left( \Alphaq(t,\x) \right)^\top \; \Tens(\x) \; \Alphaq(t,\x) + \frac{\alphap(t,\x)^2}{\rho(\x)} \right] \dd \x.
\end{equation}
Let us define the discrete Hamiltonian $\Ham^d(t) \eqdef \Ham(\Alphaq^d(t),\alphap^d(t))$ and the two errors:
\begin{itemize}
\item $\mathcal E^{\mathcal X}(t) \eqdef \nol \matl \Alphaq(t) \\ \alphap(t) \matr - \matl \Alphaq^d(t) \\ \alphap^d(t) \matr \nor_{\mathcal X}$ the absolute error in a suitable energy space $\mathcal X$;
\item
$\mathcal E^\Ham(t) \eqdef \Ham(t) - \Ham^d(t)$ the error between the continuous and the discrete Hamiltonians.
\end{itemize}
Under suitable assumptions (regularity, conformity and order given by a parameter $\kappa$) on the three finite element families (for $\Alphaq$ in $\Omega$, $\alphap$ in $\Omega$, and $(u, y)$ on $\partial\Omega$), for all $T>0$, all initial data smooth enough and all $u$ smooth enough, there exist $C>0$, independent of $h$ the mesh size parameter, and $h^*>0$ such that
$$
\mathcal E^{\mathcal X}(t) \le C ~ h^{\kappa}, \Forall t \in [0,T], \; h \in (0, h^*).
$$
Furthermore, under \emph{compatibility} assumptions between the three finite element families, the Hamiltonian error satisfies
$$
\mathcal E^{\Ham}(t) - \mathcal E^{\Ham}(0) = \frac{1}{2} \left( \left( \mathcal E^{\mathcal X}(t) \right)^2 - \left( \mathcal E^{\mathcal X}(0) \right)^2 \right), \Forall t \in [0,T],
$$
giving a convergence rate of order $2\kappa$.
\end{theorem*}

\subsection{Organization of the paper}

The paper is organized as follows: in Section~\ref{Sec:Wave}, the well-posedness of the physical system~\eqref{eq:waves-continuous}--\eqref{eq:waves-boundary} is recalled, and some assumptions on the regularity of the solutions are made. In Section~\ref{Sec:PFEM}, the PFEM is applied and discussed, and the resulting finite-dimensional Dirac structure is highlighted. The case of Dirichlet boundary control, \ie \emph{switching} control and observation, is addressed in Section~\ref{Sec:switched-system}. In Section~\ref{Sec:AnaNum}, the main convergence results are proved for a general Galerkin approximation method, namely Theorem~\ref{th:General-State} for the state error,  and  Theorem~\ref{th:General-Hamiltonian} for the Hamiltonian error. In Section~\ref{Sec:AnaNum-PkRTl}, accurate combinations of finite elements are proposed for a given rate of convergence, by minimizing the number of degrees of freedom. In Section~\ref{Sec:SimuNum}, 2D simulations are provided to exhibit the proven convergence rates and its optimality (\ie maximizing the convergence rate with the minimal number of degrees of freedom); several test cases are provided to illustrate the flexibility of the method. Finally, Section~\ref{Sec:Conclusion} concludes this work with a summary of the results and draws some  perspectives.

\section{The \texorpdfstring{$N$}{N}-dimensional wave equation as a pHs}\label{Sec:Wave}

In this section, the boundary-controlled-and-observed wave system~\eqref{eq:waves-continuous}--\eqref{eq:waves-boundary} is firstly recast as a port-Hamiltonian system. This system has already been studied in the pHs framework in~\cite{KurZwa15}, in a more general context, \ie with several boundary conditions on a partition of $\partial\Omega$ and internal fluid damping. Secondly, well-posedness is recalled~\cite{KurZwa15} and a refined regularity result is conjectured, assuming a higher regularity of the physical parameters, the initial data, and the control.

Although it should be possible to prove the regularity assumptions making use of Boundary Control Systems theory~\cite[Chapter~10.]{TucWei09} or by adapting the results in~\cite{KurZwa15} to the uniform boundary control case considered here, it goes beyond the scope of this work.

\subsection{The distributed-parameters port-Hamiltonian system}\label{Sec:Wave-as-pHs}

From now on, $H^\kappa(\Omega)$ denotes the usual Sobolev space for $\kappa \in \R$, and $H^0(\Omega)$ is identified with $L^2(\Omega)$ (the same notations are used on the boundary $\partial\Omega$). We also write $\L^2(\Omega) \eqdef (L^2(\Omega))^N$ and $\H^\kappa(\Omega) \eqdef (H^\kappa(\Omega))^N$.

\begin{definition}[Traces {\cite[Chapter~2]{Ces96}}]\label{def:traces}
The linear trace operators are defined as follows
\begin{itemize}
\item the Dirichlet trace operator $\gamma_0$, defined by $\gamma_0 (v) \eqdef v_{|\partial\Omega}$ for $v \in \mathcal C^\infty(\overline\Omega)$, extends continuously from $H^1(\Omega)$ onto $H^{\frac{1}{2}} (\partial \Omega)$;
\item the normal trace operator $\gamma_\perp$, defined by $\gamma_\perp (\v) \eqdef (\v \; \n)_{|\partial\Omega}$ on $(\mathcal C^\infty(\overline\Omega))^N$, extends continuously from $\H(\div ; \Omega) \eqdef \left\{ \v \in \L^2(\Omega) \; \mid \; \div\left(\v\right) \in L^2(\Omega) \right\}$ onto $H^{-\frac{1}{2}}(\partial\Omega)$.
\end{itemize}
\end{definition}

The so-called Green's formula then reads: for all $\v \in \H(\div ; \Omega)$, $v \in H^1(\Omega)$,
\begin{equation}\label{eq:Green}
\int_\Omega v(\x) \div(\v(\x)) \dd \x = - \int_\Omega \left(\grad (v(\x))\right)^\top \; \v(\x) \dd \x 
+ \psl \gamma_\perp (\v), \gamma_0 (v) \psr_{H^{-\frac{1}{2}}(\partial\Omega), H^{\frac{1}{2}}(\partial\Omega)}.
\end{equation}
The last term in~\eqref{eq:Green} is the duality bracket between $H^{\frac{1}{2}}(\partial\Omega),$ and $H^{-\frac{1}{2}}(\partial\Omega)$. Note that as soon as $\gamma_\perp (\v) \in L^2(\partial\Omega)$, this bracket reduces to the usual $L^2(\partial\Omega)$-inner product~\cite[§~2.9]{TucWei09}.

Let us define the \emph{strain} $\Alphaq \eqdef \grad (w)$ and the \emph{linear momentum} $\alphap \eqdef \rho \; \partial_t w$. Then one can rewrite the first line of System~\eqref{eq:waves-continuous} as
\begin{equation}\label{eq:Hamilton-System}
\matl \partial_t \Alphaq \\ \partial_t \alphap \matr = \underbrace{\matl 0 & \grad \\ \div & 0 \matr}_{=: \mathcal J} \underbrace{\matl \Tens & 0 \\ 0 & \rho^{-1} \matr}_{=: \mathcal Q} \matl \Alphaq \\ \alphap \matr.
\end{equation}

The boundary control and observation~\eqref{eq:waves-boundary} then read
\begin{equation}\label{eq:Hamilton-System-Boundary}
u = \gamma_\perp \left( \Tens \; \Alphaq \right), \qquad y = \gamma_0 \left(\rho^{-1} \alphap \right).
\end{equation}

We already define in~\eqref{eq:Hamiltonian} the Hamiltonian of~\eqref{eq:Hamilton-System}--\eqref{eq:Hamilton-System-Boundary}
$$
\Ham(t) \eqdef \Ham(\Alphaq(t),\alphap(t)) \eqdef \dfrac{1}{2} \int_\Omega \left[ \left( \Alphaq(t,\x) \right)^\top \; \Tens(\x) \; \Alphaq(t,\x) + \frac{\alphap(t,\x)^2}{\rho(\x)} \right] \dd \x,
$$
corresponding to the sum of the potential and kinetic energy, \ie the total mechanical energy of the system. Making use of Green's formula~\eqref{eq:Green} together with~\eqref{eq:Hamilton-System}--\eqref{eq:Hamilton-System-Boundary}, one gets that for all $t\ge0$
\begin{equation}\label{eq:variation-Hamiltonian}
\ddt \Ham(t) = \psl u(t), y(t) \psr_{H^{-\frac{1}{2}}(\partial\Omega), H^{\frac{1}{2}}(\partial\Omega)},
\end{equation}
meaning that the variation of energy is the power supplied to the system at the boundary~\cite{SchMas02}.

\begin{remark}
In the Hamiltonian formalism, $\Alpha$ are called the \emph{energy variables} while $\vector{e} := \delta_{\Alpha} \Ham(\Alpha)$, the variational derivative of $\Ham$ with respect to $\Alpha$~\cite{SchMas02} are the \emph{co-energy variables}. The relations between $\vector{e}$ and $\Alpha$, linear in the present case, are known as the \emph{constitutive relations}, which enable to close the system of equations.
\end{remark}

\begin{remark}
The \emph{co-energy variables} $\vector{e}_q \eqdef \delta_{\Alphaq} \Ham$ and $e_p \eqdef \delta_{\alphap} \Ham$ are physically meaningful: $\vector{e}_q = \Tens \; \Alphaq = \Tens \; \grad \left( w \right)$ is the \emph{stress}, and $e_p = \rho^{-1} \alphap = \partial_t w$ is the \emph{deflection velocity}. Furthermore, as seen in~\eqref{eq:variation-Hamiltonian}, a relevant way to control and observe the system through the boundary is related to both the traces of these co-energy variables.
\end{remark}

\subsection{Existence and uniqueness of solutions}
\label{Sec:existence-uniqueness}

Let $\mathcal X \eqdef \L^2(\Omega) \times L^2(\Omega)$ be the energy space, endowed with the inner product
$$
\begin{array}{ll}
\dsp \psl \z_1, \z_2 \psr_{\mathcal X} 
&\eqdef\dsp \psl \mathcal Q \matl {\Alphaq}_1 \\ {\alphap}_1 \matr, \matl {\Alphaq}_2 \\ {\alphap}_2 \matr \psr_{\L^2(\Omega) \times L^2(\Omega)} \\
&=\dsp \int_\Omega \left( \left(\Tens(\x) \; {\Alphaq}_1(\x)\right)^\top \; {\Alphaq}_2(\x) + \left( \rho(\x)^{-1} {\alphap}_1(\x) \right) {\alphap}_2(\x) \right) \dd \x,
\end{array}
$$
for all $(\z_1, \z_2) \eqdef \left( \matl {\Alphaq}_1 \\ {\alphap}_1 \matr, \matl {\Alphaq}_2 \\ {\alphap}_2 \matr \right) \in \mathcal X^2$. It is clear from the assumption on $\rho$ and $\Tens$ that the norm inherited from this inner product is equivalent to the usual $\L^2(\Omega) \times L^2(\Omega)$-norm.

In addition, let $\mathcal Z \eqdef \mathcal Q^{-1} \left[ \begin{matrix} \H(\div ; \Omega) \\ H^1(\Omega) \end{matrix}\right]$ be the solution space, $\mathcal U \eqdef H^{-\frac{1}{2}}(\partial\Omega)$ the control space and $\mathcal Y \eqdef \mathcal U' = H^{\frac{1}{2}}(\partial\Omega)$ the observation space. It has been shown in~\cite[Corollary~4.3]{KurZwa15} that this leads to an internally well-posed strong impedance conservative boundary control system on $(\mathcal U,\mathcal X,\mathcal Y)$.
\begin{theorem}[Corollary~4.3 in \cite{KurZwa15}]\

For all $u \in \mathcal C^2([0,\infty);\mathcal U)$, $\z_0 \eqdef \matl {\Alphaq}_0 \\ {\alphap}_0 \matr \eqdef \matl \grad (w_0) \\ \rho^{-1} w_1 \matr \in \mathcal Z$ such that $u(0) = \gamma_\perp \left( \Tens \; \grad (w_0) \right)$, there exists a unique solution to~\eqref{eq:Hamilton-System}--\eqref{eq:Hamilton-System-Boundary} with
$$
\z = \matl \Alphaq \\ \alphap \matr \in \mathcal C^1([0,\infty); \mathcal X) \cap \mathcal C([0,\infty); \mathcal Z), \qquad y \in \mathcal C([0,\infty); \mathcal Y).
$$
\end{theorem}

\begin{remark}
The compatibility condition $u(0) = \gamma_\perp \left( \Tens \; \grad (w_0) \right)$ is well-known in boundary control systems theory. See for instance \cite[Chapter~10]{TucWei09} for more details.
\end{remark}

With this material at hand, we are now able to define properly what is meant by {\em formally skew-symmetric}: the operator $\mathcal J \mathcal Q$ restricted to $\H_0(\div ; \Omega) \times H^1(\Omega)$, where $\H_0(\div ; \Omega) \eqdef \left\{ \v \in \H(\div ; \Omega) \; \mid \; \gamma_\perp \left( \v \right) = 0 \right\}$, is skew-adjoint on $\mathcal X$. This follows from Green's formula~\eqref{eq:Green}.

\vspace{5mm}
As is often the case in numerical analysis, sufficient regularity of the solution of the continuous problem is required to be allowed to use the interpolation error inequalities and prove convergence. The following development is formal, and its usefulness for the rest of this work will be enlightened in Remark~\ref{rem:utility-regularity}.

For all integer $\kappa\ge0$, assume that $\partial\Omega$ is $\mathcal C^{\kappa+2}$ and define
\begin{itemize}
\item
$\H^0(\div ; \Omega) \eqdef \L^2(\Omega)$ and
$\H^\kappa(\div ; \Omega) \eqdef \left\{ \v \in \H^{\kappa-1}(\Omega) \; \mid \; \div\left(\v\right) \in H^{\kappa-1}(\Omega) \right\}$ when $\kappa \ge 1$, endowed with the inner product
$$
\psl \v_1, \v_2 \psr_{\H^\kappa(\div ; \Omega)} \eqdef \psl \v_1, \v_2 \psr_{\H^{\kappa-1}(\Omega)} + \psl \div (\v_1), \div (\v_2) \psr_{H^{\kappa-1}(\Omega)};
$$
\item
$\mathcal X_\kappa \eqdef \H^\kappa(\div ; \Omega) \times H^\kappa(\Omega)$ endowed with the bilinear form, for all $\z_i := \matl \v_i \\ v_i \matr \in \mathcal X_\kappa$, $i = 1, 2$
$$
\psl \z_1, \z_2 \psr_{\mathcal X_\kappa} \eqdef \psl \mathcal Q \z_1, \z_2 \psr_{\H^\kappa(\div ; \Omega) \times H^\kappa(\Omega)} := \psl \mathcal \Tens \; \v_1, \v_2 \psr_{\H^\kappa(\div ; \Omega)} + \psl \mathcal \rho^{-1} \; v_1, v_2 \psr_{H^\kappa(\Omega)},
$$
the energy space;
\item
$\mathcal Z_\kappa \eqdef \mathcal Q^{-1} \left[ \begin{matrix} \H^{\kappa+1}(\div ; \Omega) \\ H^{\kappa+1}(\Omega) \end{matrix} \right]$ the solution space;
\item
$\mathcal U_\kappa \eqdef H^{\kappa-\frac{1}{2}}(\partial\Omega)$ the control space;
\item
$\mathcal Y_\kappa \eqdef H^{\kappa+\frac{1}{2}}(\partial\Omega)$ the observation space.
\end{itemize}

It is known from~\cite[Chapter~2; Theorem~1 \& Proposition~10]{Ces96} that the traces of Definition~\ref{def:traces} satisfy
\begin{itemize}
\item $\gamma_0$ is continuous from $H^{\kappa+1}(\Omega)$ onto $H^{\kappa+\frac{1}{2}} (\partial \Omega)$;
\item $\gamma_\perp$ is continuous from $\H^{\kappa+1}(\div ; \Omega)$ onto $H^{\kappa-\frac{1}{2}}(\partial\Omega)$.
\end{itemize}

Assume furthermore that $\rho$ and $\Tens$ are smooth enough for $\mathcal X_\kappa$ to be  a Hilbert space.

\begin{conjecture}\label{conj:regularity}
With the above notations and assumptions, it holds
\begin{multline}\label{H0}\tag{H0}
\forall \; \z_0 \eqdef \matl \grad (w_0) \\ \rho^{-1} w_1 \matr \in \mathcal Z_\kappa, \; \forall \; u \in \mathcal C^2([0,\infty);\mathcal U_\kappa) \; : \; u(0) = \gamma_\perp \left( \Tens \; \grad (w_0) \right), \\
\text{there exists a unique solution } \; \z \in \mathcal C^1([0,\infty);\mathcal X_\kappa) \cap \mathcal C([0,\infty);\mathcal Z_\kappa), \text{ with } \; y \in \mathcal C([0,\infty);\mathcal Y_\kappa).
\end{multline}
\end{conjecture}

Roughly speaking, it claims that increasing the regularity on $\rho$, $\Tens$, $\partial\Omega$, $w_0$, $w_1$, and $u$, increases the space regularity of solutions (using the continuity and surjectivity of $\gamma_0$ and $\gamma_\perp$). This seems legitimate according to \cite[Corollary~4.3]{KurZwa15} (which includes the above case $\kappa=0$). Although this is not proved, this seems to be a reasonable conjecture at the mathematical level, \ie for existence and uniqueness of smooth solutions.

\begin{remark}\label{rem:utility-regularity}
The main purpose of~\eqref{H0} is to provide a relation between the maximal $\H^{\ell}(\div ; \Omega)$- and $H^k(\Omega)$-regularities of $\Alphaq$ and $\alphap$ respectively, allowing for an optimal choice of the order of the finite-dimensional spaces of approximation. Finally, it has to be kept in mind that Conjecture~\ref{conj:regularity} implies that if $\kappa$ is supposed to be the maximal regularity (in space) of $\Alphaq$, \ie $\Alphaq(t) \in \Tens^{-1} \; \H^\kappa(\div ; \Omega)$ but $\Alphaq(t) \not\in \Tens^{-1} \; \H^{\kappa+1}(\div ; \Omega)$, then $\kappa$ is also the maximal regularity (in space) of $\alphap$, \ie $\alphap(t) \in \rho H^\kappa(\Omega)$ but $\alphap(t) \not\in \rho H^{\kappa+1}(\Omega)$, and reciprocally.
\end{remark}

\subsection{The weak co-energy formulation}
\label{Sec:weak-formulation}

As it is intended to apply a conforming finite element method, it is mandatory to use a formulation which enables the use of available finite elements in numerical softwares without destroying the sparsity property of the FEM, \eg avoiding matrix inversion.

A simple way to achieve this is to work on the \emph{co-energy formulation}, which transfers the physical parameters from the right-hand side of~\eqref{eq:Hamilton-System} to its left-hand side, by inverting them \emph{at the continuous level}. More precisely, rewriting the initial system~\eqref{eq:Hamilton-System} by making use of the relations $\Alphaq = \Tens^{-1} \; \vector{e}_q$ and $\alphap = \rho e_p$, one obtains the equivalent port-Hamiltonian system, known as the \emph{co-energy formulation}
\begin{equation}\label{eq:Co-energy-Hamilton-System}
\matl \Tens^{-1} & 0 \\ 0 & \rho \matr \matl \ddt \vector{e}_q \\ \ddt e_p \matr = \mathcal J \matl \vector{e}_q \\ e_p \matr, \qquad u = \gamma_\perp \left( \vector{e}_q \right), \qquad y = \gamma_0 \left( e_p \right).
\end{equation}
At the discrete level, this gives rise to weighted mass matrices, carrying \emph{all} the physical parameters.

Multiplying~\eqref{eq:Co-energy-Hamilton-System} in $\L^2(\Omega) \times L^2(\Omega)$ by arbitrary test functions $\matl \v_q \\ v_p \matr$, one gets
$$
\psl \matl \Tens^{-1} & 0 \\ 0 & \rho \matr \matl \partial_t \vector{e}_q \\ \partial_t e_p \matr, \matl \v_q \\ v_p \matr \psr_{\L^2(\Omega) \times L^2(\Omega)} = \psl \matl 0 & \grad \\ \div & 0 \matr \matl \vector{e}_q \\ e_p \matr, \matl \v_q \\ v_p \matr \psr_{\L^2(\Omega) \times L^2(\Omega)}.
$$
which also reads
$$
\left\lbrace
\begin{array}{rl}
\psl \partial_t \vector{e}_q, \Tens^{-1} \; \v_q \psr_{\L^2(\Omega)} &= \psl \grad \left( e_p \right), \v_q \psr_{\L^2(\Omega)},\\
\psl \partial_t e_p, \rho \, v_p \psr_{L^2(\Omega)} &= \psl \div \left( \vector{e}_q \right), v_p \psr_{L^2(\Omega)}.
\end{array}
\right.
$$
At this stage, the boundary control in~\eqref{eq:Co-energy-Hamilton-System} does not appear in the formulation yet. To this end, we apply Green's formula~\eqref{eq:Green} on the second line only and obtain
\begin{equation}\label{eq:FV-continuous}
\left\lbrace
\begin{array}{rl}
\dsp \psl \partial_t \vector{e}_q, \Tens^{-1} \; \v_q \psr_{\L^2(\Omega)} &=\dsp \psl \grad \left( e_p \right), \v_q \psr_{\L^2(\Omega)},\\
\dsp \psl \partial_t e_p, \rho \, v_p \psr_{L^2(\Omega)} &=\dsp - \psl \vector{e}_q, \grad \left( v_p \right) \psr_{\L^2(\Omega)} + \psl u, \gamma_0 \left( v_p \right) \psr_{\mathcal U, \mathcal Y},
\end{array}
\right.
\end{equation}
remembering that $\mathcal U = H^{-\frac{1}{2}}(\partial \Omega)$ and $\mathcal Y = H^{\frac{1}{2}}(\partial\Omega)$. These equations make sense if $\v_q \in \L^2(\Omega)$ and $v_p \in H^1(\Omega)$. Altogether, the test functions have to belong to $\L^2(\Omega) \times H^{1}(\Omega)$ (note that this is neither $\mathcal X$ nor $\mathcal Z$ defined in Section~\ref{Sec:existence-uniqueness}).

\section{Structure-preserving discretisation}
\label{Sec:PFEM}

The aim of this section is to show how a simple integration by part on a \emph{partition} of the state space, as in the MFEM, is able to transforms a distributed-parameters port-Hamiltonian system (\ie infinite-dimensional) into a lumped-parameters port-Hamiltonian system (\ie finite-dimensional). The main difference with the MFEM is the non-homogeneous boundary condition (\ie the boundary control) applied to the partition to be integrated by part. The major interest of this scheme, known as the Partitioned Finite Element Method (PFEM)~\cite{CarMatLef19}, is that it directly leads to a discrete version of the power-balance satisfied by the discrete Hamiltonian, defined as the continuous one evaluated on the approximations of energy variables.

\subsection{Discrete weak formulation and matrix form}
\label{Sec:discrete-weak-formulation}

We are now in a position to discretize the system in space. Assume that we have at our disposal three finite dimensional spaces, typically given by finite elements respectively $\L^2(\Omega)$-conforming, $H^1(\Omega)$-conforming and $H^\frac{1}{2}(\partial\Omega)$-conforming
$$
\H_q \eqdef \span \left\{ \left( \phiq^i \right)_{i=1,\dots,N_q} \right\} \subset \L^2(\Omega) \quad \text{ of dimension } N_q \in \N,
$$
$$
H_p \eqdef \span \left\{ \left( \phip^k \right)_{k=1,\dots,N_p} \right\} \subset H^1(\Omega) \quad \text{ of dimension } N_p \in \N,
$$
and
$$
H_\partial \eqdef \span \left\{ \left( \psi^m \right)_{m=1,\dots,N_\partial}\right\} \subset H^\frac{1}{2}(\partial\Omega) \quad \text{ of dimension } N_\partial \in \N.
$$

\begin{remark}
Note that the boundary basis $\left( \psi^m \right)_{m=1,\dots,N_\partial}$, used to approximate both $u$ and $y$ for simplicity, is chosen to be $H^{\frac{1}{2}}(\partial\Omega)$-conforming, while $u$ only required $H^{-\frac{1}{2}}(\partial\Omega)$ to be approximated in a conforming manner. This is a convenient assumption which leads to the usual $L^2(\partial\Omega)$-inner product at the boundary, without loss of generality as soon as $\kappa \ge 1$. 
\end{remark}

Let us approximate the function $\vector{e}_q$ in $\H_q$ by
$$
\vector{e}_q(t,\x) \simeq \vector{e}_q^d(t,\x) \eqdef \sum_{i=1}^{N_q} e_q^i(t) \phiq^i(\x) = \left(\Phiq(\x)\right)^\top \; \equ(t), \Forall t\ge0, \x \in \Omega,
$$
where for all $t\ge0$ and $\x\in\Omega$, we introduce the compact notations
$$
\Phiq(\x) \eqdef \matl \left(\phiq^1(\x)\right)^\top \\ \vdots \\ \left(\phiq^{N_q}(\x)\right)^\top \matr \in \R^{N_q \times N}, \qquad \qquad \equ(t) \eqdef \matl e_q^1(t) \\ \vdots \\ e_q^{N_q}(t) \matr \in \R^{N_q}.
$$
In the same way, $e_p$ is approximated in $H_p$ by
$$
e_p(t,\x) \simeq e_p^d(t,\x) \eqdef \sum_{k=1}^{N_p} e_p^k(t) \phip^k(\x) = \left(\Phip(\x)\right)^\top \; \epu(t), \Forall t\ge0, \x \in \Omega,
$$
where for all $t\ge0$ and all $\x\in\Omega$, we have the compact notations
$$
\Phip(\x) \eqdef \matl \phip^1(\x) \\ \vdots \\ \phip^{N_p}(\x) \matr \in \R^{N_p}, \qquad \qquad \epu(t) \eqdef \matl e_p^1(t) \\ \vdots \\ e_p^{N_p}(t) \matr \in \R^{N_p}.
$$
Finally, $u$ is approximated in $H_\partial$ by
$$
u(t,\s) \simeq u^d(t,\s) \eqdef \sum_{m=1}^{N_\partial} u^m(t) \psi^m(\s) = \left(\Psi(\s)\right)^\top \; \underline{u}(t), \Forall t\ge0, \s \in \partial\Omega,
$$
where for all $t\ge0$ and all $\s\in\partial\Omega$, we also use the compact notations
$$
\Psi(\s) \eqdef \matl \psi^1(\s) \\ \vdots \\ \psi^{N_\partial}(\s) \matr \in \R^{N_\partial}, \qquad \qquad \underline{u}(t) \eqdef \matl u^1(t) \\ \vdots \\ u^{N_\partial}(t) \matr \in \R^{N_\partial}.
$$
It is now possible to formulate the discrete variational formulation from the continuous one~\eqref{eq:FV-continuous} on $\H_q \times H_p \times H_\partial$: for all $j=1, \dots, N_q$ and all $\ell=1, \dots, N_p$, we are seeking for $(\vector{e}_q^d,e_p^d) \in \H_q \times H_p$ such that
\begin{equation}\label{eq:FV-discrete}
\left\lbrace
\begin{array}{rl}
\psl \partial_t \vector{e}_q^d, \Tens^{-1} \; \phiq^j \psr_{\L^2(\Omega)} &= \psl \grad \left( e_p^d \right), \phiq^j \psr_{\L^2(\Omega)},\\
\psl \partial_t e_p^d, \rho \, \phip^\ell \psr_{L^2(\Omega)} &= - \psl \vector{e}_q^d, \grad \left( \phip^\ell \right) \psr_{\L^2(\Omega)} + \psl u^d, \gamma_0 \left( \phip^\ell \right) \psr_{L^2(\partial\Omega)}.
\end{array}
\right.
\end{equation}
From the definition of $\vector{e}_q^d$, $e_p^d$ and $u^d$, this leads to
\begin{equation}\label{eq:system-FV-def-discrete}
\left\lbrace
\begin{array}{rl}
\dsp \sum_{i=1}^{N_q} \ddt e_q^i \psl \phiq^i, \Tens^{-1} \; \phiq^j \psr_{\L^2(\Omega)} &=\dsp \sum_{k=1}^{N_p} e^k_p \psl \grad \left( \phip^k \right), \phiq^j \psr_{\L^2(\Omega)},\\
\dsp \sum_{k=1}^{N_p} \ddt e_p^k \psl \phip^k, \rho \, \phip^\ell \psr_{L^2(\Omega)} &=\dsp - \sum_{i=1}^{N_q} e_q^i \psl \phiq^i, \grad \left( \phip^\ell \right) \psr_{\L^2(\Omega)} \\
& \dsp \hspace{10em} + \sum_{m=1}^{N_\partial} u^m \psl \psi^m, \gamma_0 \left( \phip^\ell \right) \psr_{L^2(\partial\Omega)}.
\end{array}
\right.
\end{equation}
Now, denoting
$$
\grad \left( \Phip \right) \eqdef \matl \left( \grad \left( \phip^1 \right) \right)^\top \\ \vdots \\ \left( \grad \left( \phip^{N_p} \right) \right)^\top \matr \quad \in \R^{N_p \times N},
$$
the gradient of the $p$-type family (which is $H^1(\Omega)$-conforming by hypothesis),
$$
M_{\Tens^{-1}} \eqdef \int_\Omega \Phiq(\x) \; \Tens^{-1}(\x) \; \left(\Phiq(\x)\right)^\top \dd \x \quad \in \R^{N_q \times N_q},
$$
$$
M_{\rho} \eqdef \int_\Omega \rho(\x) \, \Phip(\x) \; \left(\Phip(\x)\right)^\top \dd \x \quad \in \R^{N_p \times N_p},
$$
the mass matrices taking the metric of $\mathcal X$ into account,
$$
D \eqdef \int_\Omega \Phiq(\x) \; \left(\grad \left( \Phip(\x) \right) \right)^\top \dd \x \quad \in \R^{N_q \times N_p},
$$
the \emph{averaged gradient} and
$$
B_\partial \eqdef \int_{\partial\Omega} \gamma_0 \left( \Phip \right)(\s) \; \left( \Psi(\s) \right)^\top \dd \s \quad \in \R^{N_p \times N_\partial},
$$
the discrete boundary control operator, we get the following finite-dimensional dynamical system from~\eqref{eq:system-FV-def-discrete}
\begin{equation}\label{eq:dim-finie}
\matl M_{\Tens^{-1}} & 0 \\ 0 & M_{\rho} \matr \; \ddt \matl \equ(t) \\ \epu(t) \matr = \matl 0 & D \\ -D^\top & 0 \matr \; \matl \equ(t) \\ \epu(t) \matr + \matl 0 \\ B_\partial \matr \; \underline{u}(t), \Forall t\ge0.
\end{equation}
Finally, defining
$$
M_\partial \eqdef \int_{\partial\Omega} \Psi(\s) \; (\Psi(\s))^\top \dd \s \quad \in \R^{N_\partial \times N_\partial},
$$
the boundary mass matrix and
$$
\mathcal B_\partial \eqdef \matl 0 \\ B_\partial \matr \quad \in \R^{(N_q + N_p) \times N_\partial},
$$
the extended boundary control operator, the output is given for all $t\ge0$ by
\begin{equation}\label{eq:discrete-observation}
M_\partial \; \underline{y}(t) \eqdef \mathcal B_\partial^\top \; \matl \equ(t) \\ \epu(t) \matr = B_\partial^\top \; \epu(t).
\end{equation}
Now, system~\eqref{eq:dim-finie}--\eqref{eq:discrete-observation} is a finite-dimensional port-Hamiltonian system.

\begin{remark}
One can gather equations~\eqref{eq:dim-finie}--\eqref{eq:discrete-observation} under the flows--efforts formulation often used in the port-Hamiltonian systems community
\begin{equation}\label{eq:flows-efforts}
\matl M_{\Tens^{-1}} & 0 & 0 \\ 0 & M_{\rho} & 0 \\ 0 & 0 & M_\partial \matr \; \matl \ddt \equ(t) \\ \ddt \epu(t) \\ - \underline{y}(t) \matr = \matl 0 & D & 0 \\ -D^\top & 0 & B_\partial \\ 0 & -B_\partial^\top & 0 \matr \; \matl \equ(t) \\ \epu(t) \\ \underline{u}(t) \matr.
\end{equation}
The block diagonal symmetric positive-definite matrix constituted by the mass matrices on the left-hand side carry the physical parameters by taking the induced metric into account. On the right-hand side, the skew-symmetric matrix is known as the \emph{extended structure matrix}. Together, they give a \emph{kernel representation}~\cite{SchJel14} of the underlying Dirac structure~\cite{SerMatHai19d}.
\end{remark}

At this stage, the co-energy formulation~\eqref{eq:Co-energy-Hamilton-System} of the infinite-dimensional port-Hamiltonian system~\eqref{eq:Hamilton-System}--\eqref{eq:Hamilton-System-Boundary} has been accurately discretized as a finite-dimensional port-Hamiltonian system~\eqref{eq:dim-finie}--\eqref{eq:discrete-observation}. The next step is to define a discrete version of the Hamiltonian $\Ham$ in order to perfectly mimic the power balance~\eqref{eq:variation-Hamiltonian}.

\begin{definition}\label{def:discrete-Hamiltonian}
The discrete Hamiltonian $\Ham^d$ is defined as the evaluation of $\Ham$ on the approximations $\Alphaq^d \eqdef \Tens^{-1} \; \vector{e}_q^d$, and $\alphap^d \eqdef \rho \, e_p^d$, namely
$$
\Ham^d(t) \eqdef \Ham(\Alphaq^d(t), \alphap^d(t)).
$$
\end{definition}

\begin{proposition}\label{prop:discrete-power-balance}
The discrete Hamiltonian $\Ham^d$ reads
\begin{equation}\label{eq:Hamilton-System-Discrete}
\Ham^d(t) = \frac{1}{2} \left(\equ(t)\right)^\top \; M_{\Tens^{-1}} \; \equ(t) 
+ \frac{1}{2} \left(\epu(t)\right)^\top \; M_{\rho} \; \epu(t).
\end{equation}
For all $t\ge0$, the following power balance holds
\begin{equation}\label{eq:variation-Hamiltonian-discrete}
\begin{array}{ll}
\dsp
\dsp \ddt \Ham^d(t) 
&=\dsp 
\left( \underline{u}(t) \right)^\top \; M_{\partial} \; \underline{y}(t), \\
&=\dsp 
\psl u^d(t), y^d(t) \psr_{L^2(\partial\Omega)},
\end{array}
\end{equation}
which is the discrete counterpart of~\eqref{eq:variation-Hamiltonian}.
\end{proposition}

\begin{proof}
The equality~\eqref{eq:Hamilton-System-Discrete} is straightforward.

Clearly
$$
\frac{\dd}{\dd t} \Ham^d(t) =\left( \equ(t) \right)^\top \; M_{\Tens^{-1}} \; \left( \ddt \equ(t) \right) + \left( \epu (t) \right)^\top \; M_{\rho} \; \left( \ddt \epu(t) \right),
$$
thanks to the symmetry of the mass matrices.

Multiplying~\eqref{eq:flows-efforts} by $\matl \equ(t) \\ \epu(t) \\ \underline{u}(t) \matr$ on the left leads to
$$
\left( \equ(t) \right)^\top \; M_{\Tens^{-1}} \; \left( \ddt \equ(t) \right) + \left( \epu (t) \right)^\top \; M_{\rho} \; \left( \ddt \epu(t) \right) = \left(\underline{u}(t)\right)^\top \; M_\partial \; \underline{y}(t),
$$
thanks to the skew-symmetry of the extended structure matrix, and the result follows.
\end{proof}

\subsection{Other causalities}
\label{Sec:switched-system}

The proposed strategy can handle other causalities, \ie other collocated boundary control and observation, in a straightforward manner. Let us focus on the other uniform causality, \ie with deflection velocity control.

It only consists on switching the role played by $u$ and $y$, \ie replace~\eqref{eq:waves-boundary} by
\begin{equation}\label{eq:waves-boundary-switch}\tag{\ref{eq:waves-boundary}S}
\left\{\begin{array}{ll}
\widetilde u(t,\x) = \partial_t w(t,\x),& \Forall \x \in \partial\Omega, t \ge 0, \\
\widetilde y(t,\x) = \left( \Tens(\x) \; \grad (w(t,\x)) \right)^\top \; \n(\x),& \Forall \x \in \partial\Omega, t \ge 0,
\end{array}\right.
\end{equation}
leading to
\begin{equation}\label{eq:FV-continuous-switch}\tag{\ref{eq:FV-continuous}S}
\left\lbrace
\begin{array}{rl}
\dsp \psl \partial_t \vector{e}_q, \Tens^{-1} \; \v_q \psr_{\L^2(\Omega)} &=\dsp - \psl e_p, \div \left( \v_q \right) \psr_{L^2(\Omega)} + \psl \gamma_\perp \left( \v_q \right), \widetilde u \psr_{L^2(\partial\Omega)},\\
\dsp \psl \partial_t e_p,\rho \, v_p \psr_{L^2(\Omega)} &=\dsp \psl \div \left( \vector{e}_q \right), v_p \psr_{L^2(\Omega)},
\end{array}
\right.
\end{equation}
instead of~\eqref{eq:FV-continuous}. The PFEM would then provide the following matrices
$$
\widetilde D = - \int_\Omega \div \left( \Phiq(\x) \right) \; \left( \Phip(\x) \right)^\top \dd \x, \quad \widetilde B_\partial = \int_{\partial\Omega} \Psi(\s) \; \left( \gamma_\perp \left( \Phiq \right)(\s)\right)^\top \dd \s,
$$
such that
$$
\matl M_{\Tens^{-1}} & 0 & 0 \\ 0 & M_{\rho} & 0 \\ 0 & 0 & M_\partial \matr \matl \ddt \equ(t) \\ \ddt \epu(t) \\ - \underline{\widetilde{y}}(t) \matr = \matl 0 & \widetilde D & \widetilde B_\partial \\ -\widetilde{D}^\top & 0 & 0 \\ -\widetilde{B}_\partial^\top & 0 & 0 \matr \matl \equ(t) \\ \epu(t) \\ \underline{\widetilde{u}}(t) \matr
$$
In this case, $\H_q$ is chosen $\H(\div ; \Omega)$-conforming and $H_p$ only $L^2(\Omega)$-conforming.

\begin{remark}
More complex causalities, such as mixed boundary controls, boundary damping, etc., can be handled in the same manner, as it will be done in Section~\ref{Sec:Simu-damping} for absorbing boundary condition. We refer to~\cite{SerMatHai19a,BruCarHaiKot20} for more details.
\end{remark}

\section{Numerical analysis}
\label{Sec:AnaNum}

This section is the core of this work: we state the main theorems and provide their proof. They are given under usual assumptions for Galerkin methods, gathered below from \eqref{H1} to \eqref{H5}, that prove classical for the finite element method (see \eg \cite{BofBreFor13,Gat14}). The aim of such an abstract numerical analysis is to propose a general framework to deal with several kinds of approximation families at the same time. This allows for recent developments such that, \eg, conforming discontinuous Galerkin elements on rectangular mesh~\cite{FenLiuWanZha21}, or even for meshfree methods~\cite{OliPor16}. Next, Section~\ref{Sec:AnaNum-PkRTl} shall focus on  well-known examples of suitable choices of  such families of conforming finite elements.

Let us consider the following
\begin{itemize}
\item $\mathcal E^{\mathcal X}(t) \eqdef \nol \matl \Alphaq(t) \\ \alphap(t) \matr - \matl \Alphaq^d(t) \\ \alphap^d(t) \matr \nor_{\mathcal X}$ the absolute error in $\mathcal X$ between the continuous and the discrete energy variables;
\item
$\mathcal E^\Ham(t) \eqdef \Ham(t) - \Ham^d(t)$ the error between the continuous and the discrete Hamiltonians.
\end{itemize}
The aim is to analyse the asymptotic behaviour of those errors, when the values of $N_q$, $N_p$ and $N_\partial$ tends towards $\infty$ (\eg when the mesh size parameter tends towards $0$). Furthermore, the best trade-off between the discretization orders of $\H_q$, $H_p$ and $H_\partial$ are provided: the number of degrees of freedom is minimized for each fixed desired convergence rate.

\begin{remark}
Thanks to~\eqref{eq:variation-Hamiltonian-discrete}, it holds $\mathcal E^\Ham(t) \eqdef \Ham(t) - \Ham^d(t) = \Ham(0) - \Ham^d(0)$ for all $t\ge0$, as soon as the system is closed (\ie with $u \equiv 0$). This result is well-known since several decades using the MFEM~\cite{Jol03}.
\end{remark}

\begin{remark}
Theorem~\ref{th:General-State} could be obtained by considering the homogeneous isotropic case only (corresponding to an identification between energy and co-energy variables). However, anisotropy and heterogeneity induce mandatory modifications about the compatibility conditions allowing for preserving de Rham cohomology, as will be shown in Theorem~\ref{th:General-Hamiltonian}.
\end{remark}

\subsection{Notations, hypotheses and basic properties}

In the sequel, the following general hypotheses are assumed. These assumptions are made of 
\begin{itemize}
\item usual Galerkin estimates on $H^1$ and $\L^2$;
\item an inverse inequality between the $H^1$- and $L^2$-norms on the \emph{finite-dimensional} space $H_p \subset H^1(\Omega)$;
\item an estimate of the $L^2$-projection in the $H^1$-norm, which proves useful to get optimality.
\end{itemize}

\subsubsection{Notations}

Let us denote
\begin{itemize}
\item $h \in(0,h^*)$ a parameter vow to tends to $0$, where $h^*>0$ (\ie $h$ is small enough);
\item $P_p$ the $L^2(\Omega)$-orthogonal projector from $L^2(\Omega)$ onto $H_p$;
\item $P_{1,p}$ the $H^1(\Omega)$-orthogonal projector from $H^1(\Omega)$ onto $H_p$;
\item $\Pb_q$ the $\L^2(\Omega)$-orthogonal projector from $\L^2(\Omega)$ onto $\H_q$.
\end{itemize}
In order to take into account the metric induced by the operator $\mathcal Q$ on $\mathcal X$, we also introduce
\begin{itemize}
\item 
$\calPb_q$ the orthogonal projector from $\L^2(\Omega)$ endowed with the weighted inner product $\psl \v_1, \Tens \; \v_2 \psr_{\L^2}$ for all $\v_1, \v_2 \in \L^2(\Omega)$, onto $\V_q \eqdef \Tens^{-1} \H_q$;
\item 
$\mathcal P_p$ the orthogonal projector from $L^2(\Omega)$ endowed with the weighted inner product $\psl v_1, \rho^{-1} v_2 \psr_{L^2}$ for all $v_1, v_2 \in L^2(\Omega)$, onto $V_p \eqdef \rho H_p$.
\end{itemize}

\subsubsection{Hypotheses}

There exists $h^*>0$ such that for all $\kappa\ge0$,
\begin{equation}\label{H1}\tag{H1}
\exists C_p > 0, \; \exists \theta_p \ge 0 \; : \; \nol P_p v_p - v_p \nor_{L^2(\Omega)} \le C_p ~ h^{\theta_p} ~ \nol v_p \nor_{H^{\kappa+1}(\Omega)}, 
\Forall v_p \in H^{\kappa+1}(\Omega), \; \forall h \in (0,h^*);
\end{equation}
\begin{multline}\label{H2}\tag{H2}
\exists C_{1,p} > 0, \; \exists \theta_{1,p} \ge 0 \; : \; \nol P_{1,p} v_p - v_p \nor_{H^1(\Omega)} \le C_{1,p} ~ h^{\theta_{1,p}} ~ \nol v_p \nor_{H^{\kappa+1}(\Omega)}, \\
\Forall v_p \in H^{\kappa+1}(\Omega), \; \forall h \in (0,h^*);
\end{multline}
\begin{equation}\label{H3}\tag{H3}
\exists C_q > 0, \; \exists \theta_q \ge 0 \; : \; \nol \Pb_q \v_q - \v_q \nor_{\L^2(\Omega)} \le C_q ~ h^{\theta_q} ~ \nol \v_q \nor_{\H^{\kappa+1}(\div ; \Omega)}, 
\Forall \v_q \in \H^{\kappa+1}(\div ; \Omega), \; \forall h \in (0,h^*);
\end{equation}
\begin{equation}\label{H4}\tag{H4}
\exists C_{1,0} > 0, \; \exists \theta_{1,0} \ge 0 \; : \; \nol \grad \left( v_p^d \right) \nor_{\L^2(\Omega)} \le C_{1,0} ~ h^{-\theta_{1,0}} ~ \nol v_p^d \nor_{L^2(\Omega)}, 
\Forall v_p^d \in H_p, \; \forall h \in (0,h^*);
\end{equation}
\begin{multline}\label{H5}\tag{H5}
\exists C_{0,1} > 0, \; \exists \theta_{0,1} \ge 0 \; : \; \nol P_p v_p - v_p \nor_{H^1(\Omega)} \le C_{0,1} ~ h^{-\theta_{0,1}} ~ \nol P_{1,p} v_p - v_p \nor_{H^1(\Omega)}, \\
\Forall v_p \in H^1(\Omega), \; \forall h \in (0,h^*).
\end{multline}

Note that in general, \eqref{H5} can be deduced from \eqref{H4} (see Lemma~\ref{lem:inverse-inequality} in Section~\ref{Sec:Appendix}). However this estimate can be strengthened (\ie $\theta_{0,1} < \theta_{1,0}$) in many cases in practice, typically with simplicial, regular and quasi-uniform meshes, one has $\theta_{0,1}=0$ thanks to the so-called Aubin-Nitsche trick~\cite{GirRav86,Sayas2004,Gat14}. It is thus given separately to ensure the optimality of the result.

\subsubsection{Basic properties}

It is important to notice the obvious properties between the projectors defined above, and denoted by straight or curly font. These will be useful in the sequel to get from the metric induced by $\mathcal Q$ to the usual one on $\left( L^2(\Omega) \right)^{N+1}$, and conversely.
\begin{itemize}
\item 
$\Tens^{-1} \Pb_q \Tens$ is a projector from $\L^2(\Omega)$ endowed with the inner product $\psl \v_1, \Tens \; \v_2 \psr_{\L^2}$ for all $\v_1, \v_2 \in \L^2(\Omega)$, onto $\V_q$;
\item 
$\rho P_p \rho^{-1}$ is a projector from $L^2(\Omega)$ endowed with the inner product $\psl v_1, \rho^{-1} v_2 \psr_{L^2}$ for all $v_1, v_2 \in L^2(\Omega)$, onto $V_p$;
\item 
$\Tens \calPb_q \Tens^{-1}$ is a projector from $\L^2(\Omega)$ onto $\H_q$;
\item 
$\rho^{-1} \mathcal P_p \rho$ is a projector from $L^2(\Omega)$ onto $H_p$.
\end{itemize}
Orthogonality of $\Pb_q$ in $\L^2(\Omega)$ and $P_p$ in $L^2(\Omega)$ imply that
\begin{equation}\label{eq:relations-between-proj}
\begin{array}{rcll}
\dsp \nol \v - \Pb_q \v \nor_{\L^2} 
&\le&\dsp \nol \v - \Tens \calPb_q \Tens^{-1} \v \nor_{\L^2},
&\dsp \Forall \v \in \L^2(\Omega), \\
\dsp \nol v - P_p v \nor_{L^2} 
&\le&\dsp \nol v - \rho^{-1} \mathcal P_p \rho ~ v \nor_{L^2},
&\dsp \Forall v \in L^2(\Omega), \\
\dsp \nol \Tens^{\frac{1}{2}} \left( \v - \calPb_q \v \right) \nor_{\L^2} 
&\le&\dsp \nol \Tens^{\frac{1}{2}} \left( \v - \Tens^{-1} \Pb_q \Tens ~ \v \right) \nor_{\L^2},
&\dsp \Forall \v \in \L^2(\Omega), \\
\dsp \nol \rho^{-\frac{1}{2}} \left( v - \mathcal P_p v \right) \nor_{L^2} 
&\le&\dsp \nol \rho^{-\frac{1}{2}} \left( v - \rho P_p \rho^{-1} v \right) \nor_{L^2},
&\dsp \Forall v \in L^2(\Omega).
\end{array}
\end{equation}
Thus, for all $\matl \Alphaq \\ \alphap \matr \in \mathcal X$
\begin{equation}\label{eq:proj-calP-in-X-into-P}
\nol \matl \Alphaq \\ \alphap \matr - \matl \calPb_q & 0 \\ 0 & \mathcal P_p \matr \matl \Alphaq \\ \alphap \matr \nor_{\mathcal X} 
\le \nol \matl \Alphaq \\ \alphap \matr - \matl \Tens^{-1} \Pb_q \Tens & 0 \\ 0 & \rho P_p \rho^{-1} \matr \matl \Alphaq \\ \alphap \matr \nor_{\mathcal X},
\end{equation}
and~\eqref{H1} to~\eqref{H5} can be written for the curly projectors $\mathcal P_p$ and $\calPb_q$ in their respective metric, using the upper and lower bounds of the physical parameters $\rho$ and $\Tens$.

\subsection{From energy variables to co-energy variables}

The main results of this section are Theorem~\ref{th:General-State} for $\mathcal E^{\mathcal X}$, and Theorem~\ref{th:General-Hamiltonian} for $\mathcal E^\Ham$ giving the convergence rates in term of those in \eqref{H1}--\eqref{H5}. Theorem~\ref{th:General-Hamiltonian} for the convergence rate of $\mathcal E^\Ham$ emphasizes the interest of the well-known compatibility conditions, appearing \eg for the finite-dimensional spaces to satisfy an exact sequence, mimicking the de Rham cohomology.

As to avoid introduction of unnecessary notations, the numerical analysis will be carried out making use of the energy variables in $\mathcal X$, since this is the natural state space identified in Section~\ref{Sec:Wave}~\cite{KurZwa15}, thanks to the following (obvious) lemma.
\begin{lemma}\label{lem:energy--co-energy}
One has
$$
\mathcal E^{\mathcal X}(t) \eqdef \nol \matl \Alphaq(t) \\ \alphap(t) \matr - \matl \Alphaq^d(t) \\ \alphap^d(t) \matr \nor_{\mathcal X} 
= \nol \matl \Tens^{-\frac{1}{2}} & 0 \\ 0 & \rho^\frac{1}{2} \matr \left[ \matl \vector{e}(t) \\ e_p(t) \matr - \matl \vector{e}_q^d(t) \\ e_p^d(t) \matr \right] \nor_{\left( L^2(\Omega) \right)^{N+1}}.
$$
\end{lemma}

\begin{proof}
The proof is straightforward thanks to the constitutive relations $\Alphaq = \Tens \; \vector{e}_q$ and $\alphap = \rho^{-1} e_p$.
\end{proof}

\subsection{The state absolute error}

\begin{theorem}\label{th:General-State}
Let $\kappa>0$ be an integer and $\Omega$ be of class $\mathcal C^{\kappa+2}$. There exists a constant $C_{\mathcal X}>0$ such that for all $T>0$, all initial data $\matl {\Alphaq}_0 \\ {\alphap}_0 \matr \in \mathcal Z_{\kappa}$, all $u \in \mathcal C^2([0,\infty);\mathcal U_\kappa)$ such that $u(0) = \gamma_\perp \left( \Tens \; {\Alphaq}_0 \right)$, and all $h\in (0,h^*)$
\begin{multline}\label{eq:th-General-State}
\mathcal E^{\mathcal X}(t)
\le \nol \matl \calPb_q & 0 \\ 0 & \mathcal P_p \matr \matl {\Alphaq}_0 \\ {\alphap}_0 \matr - \matl \Alphaq^d(0) \\ \alphap^d(0) \matr \nor_{\mathcal X} 
+ C_{\mathcal X} \max \{ 1, T \} ~ h^{\theta^*} ~ \nol \matl \Alphaq \\ \alphap \matr \nor_{L^\infty([0,T];\mathcal Z_{\kappa})} \\
+ C_{\mathcal X} T ~ h^{-\theta_{1,0}} ~ \nol u - u^d \nor_{L^\infty([0,T];\mathcal U)}, \Forall t \in [0,T],
\end{multline}
where
\begin{equation}\label{eq:rate-general}
\theta^* \eqdef \min \left\{ \theta_{1,p}-\theta_{0,1} \, ; \, \theta_p-\theta_{1,0} \, ; \, \theta_q-\theta_{1,0} \right\}.
\end{equation}
\end{theorem}

\begin{proof}
For the sake of readability, two technical lemmas are proved in Appendix~\ref{Sec:Appendix}.

Remark first that from~\eqref{H0}, $\matl \Alphaq \\ \alphap \matr \in \mathcal C([0,\infty);\mathcal Z_{\kappa})$, thus estimates~\eqref{H1}--\eqref{H2}--\eqref{H3} can indeed be applied to $\Alphaq$ and $\alphap$ for all time $t \in [0,T]$.

Let us decompose
\begin{equation}\label{eq:decomposition-error-th}
\nol \matl \Alphaq \\ \alphap \matr - \matl \Alphaq^d \\ \alphap^d \matr \nor_{\mathcal X} \le 
\underbrace{\nol \matl \Alphaq \\ \alphap \matr - \matl \calPb_q & 0 \\ 0 & \mathcal P_p \matr \matl \Alphaq \\ \alphap \matr \nor_{\mathcal X}}_{\mathcal E_1} 
+ \underbrace{\nol \matl \calPb_q & 0 \\ 0 & \mathcal P_p \matr \matl \Alphaq \\ \alphap \matr - \matl \Alphaq^d \\ \alphap^d \matr \nor_{\mathcal X}}_{\mathcal E_2}.
\end{equation}

The strategy of the proof proceeds in five steps.
\begin{itemize}
\item In \textbf{step 1}, the convergence of the first term $\mathcal E_1$ on the right-hand side of~\eqref{eq:decomposition-error-th} is proved thanks to~\eqref{eq:proj-calP-in-X-into-P}.
\item In \textbf{step 2}, Lemma~\ref{lem:der-norm-square} is applied in order to get the exact value of
$$
\frac{1}{2} \frac{\dd}{\dd t} \mathcal E_2^2 = \mathcal E_2 
\frac{\dd}{\dd t}\mathcal E_2,
$$
in terms of $\L^2(\Omega)$-inner products and boundary duality bracket.
\item In \textbf{step 3}, Cauchy-Schwarz inequality and coarse bounds are used so that
$$
\mathcal E_2 \frac{\dd}{\dd t}\mathcal E_2 \le \mathcal E_3 \mathcal E_2.
$$
Dividing by $\mathcal E_2 > 0$ leads to
$$
\frac{\dd}{\dd t}\mathcal E_2 \le \mathcal E_3.
$$
\item In \textbf{step 4}, Lemma~\ref{lem:third-term} is applied in order to estimate $\mathcal E_3$.
\item Finally, in \textbf{step 5}, the inequality obtained in step 4 is integrated in time, and all the estimates are gathered to conclude.
\end{itemize}

\begin{itemize}
\item[\textbf{Step 1}]
From~\eqref{eq:proj-calP-in-X-into-P} and using the lower bound $T_-$ for $\Tens$ and the upper bound $\rho^+$ for $\rho$, one gets (with the definition of the weighted norm on $\mathcal X$)
$$
\begin{array}{rl}\nol \matl \Alphaq \\ \alphap \matr - \matl \calPb_q & 0 \\ 0 & \mathcal P_p \matr \matl \Alphaq \\ \alphap \matr \nor_{\mathcal X}
\dsp &\le \nol \matl \Alphaq \\ \alphap \matr - \matl \Tens^{-1} \Pb_q \Tens & 0 \\ 0 & \rho P_p \rho^{-1} \matr \matl \Alphaq \\ \alphap \matr \nor_{\mathcal X} \\
\dsp &= \nol \matl \Tens^{-1} \; \Tens \; \Alphaq \\ \rho \rho^{-1} \alphap \matr - \matl \Tens^{-1} \Pb_q \Tens & 0 \\ 0 & \rho P_p \rho^{-1} \matr \matl \Alphaq \\ \alphap \matr \nor_{\mathcal X} \\
\dsp &= \nol \matl \Tens^{-\frac12} \left( \Ib_q - \Pb_q \right) & 0 \\ 0 & \sqrt{\rho} \left( I_p - P_p \right) \matr \matl \Tens \; \Alphaq \\ \rho^{-1} \alphap \matr \nor_{\L^2(\Omega) \times L^2(\Omega)} \\
\dsp &\le \frac{1}{\sqrt{T_-}} \nol \Tens \; \Alphaq - \Pb_q \Tens \; \Alphaq \nor_{\L^2(\Omega)} + \sqrt{\rho^+} \nol \rho^{-1} \alphap - P_p \rho^{-1} \alphap \nor_{L^2(\Omega)}.
\end{array}
$$
Applying~\eqref{H1} with $\rho^{-1} \alphap \in H^{\kappa+1}(\Omega)$ and~\eqref{H3} with $\Tens \; \Alphaq \in \H^{\kappa+1}(\div ; \Omega)$ leads to
\begin{equation}\label{eq:first-term-th}
\mathcal E_1
\le \max \left\{ \frac{C_q}{\sqrt{T_-}} ~ h^{\theta_q}, \sqrt{\rho^+} C_p ~ h^{\theta_p} \right\} \nol \matl \Alphaq \\ \alphap \matr \nor_{\mathcal Z_{\kappa}}.
\end{equation}

\item[\textbf{Step 2}]
Applying Lemma~\ref{lem:der-norm-square} leads to
\begin{multline*}
\frac{1}{2} \frac{\dd}{\dd t} \nol \matl \calPb_q & 0 \\ 0 & \mathcal P_p \matr \matl \Alphaq \\ \alphap \matr - \matl \Alphaq^d \\ \alphap^d \matr \nor_{\mathcal X}^2 
= \psl \grad \left( \rho^{-1} \left( \alphap - \mathcal P_p \alphap \right) \right), \Tens \; \left( \calPb_q \Alphaq - \Alphaq^d \right) \psr_{\L^2(\Omega)} \\
- \psl \Tens \; \left( \Alphaq - \calPb_q \Alphaq \right), \grad \left( \rho^{-1} \left( \mathcal P_p \alphap - \alphap^d \right) \right) \psr_{\L^2(\Omega)} 
+ \psl u - u^d, \gamma_0 \left( \rho^{-1} \left( \mathcal P_p \alphap - \alphap^d \right) \right) \psr_{\mathcal U, \mathcal Y}.
\end{multline*}

\item[\textbf{Step 3}]
From Cauchy-Schwarz inequality and the continuity of the Dirichlet trace operator on $H^1(\Omega)$
\begin{multline*}
\frac{1}{2} \frac{\dd}{\dd t} \nol \matl \calPb_q & 0 \\ 0 & \mathcal P_p \matr \matl \Alphaq \\ \alphap \matr - \matl \Alphaq^d \\ \alphap^d \matr \nor_{\mathcal X}^2 
\le \nol \Tens^\frac12 \; \grad \left( \rho^{-1} \left( \alphap - \mathcal P_p \alphap \right) \right) \nor_{\L^2(\Omega)} 
\nol \Tens^\frac12 \; \left( \calPb_q \Alphaq - \Alphaq^d \right) \nor_{\L^2(\Omega)} \\
+ \nol \Tens \; \left( \Alphaq - \calPb_q \Alphaq \right) \nor_{\L^2(\Omega)}
\nol \grad \left( \rho^{-1} \left( \mathcal P_p \alphap - \alphap^d \right) \right) \nor_{\L^2(\Omega)}
+ C_D \nol u - u^d \nor_{\mathcal U} \nol \rho^{-1} \left( \mathcal P_p \alphap - \alphap^d \right) \nor_{L^2(\Omega)}  \\
+ C_D \nol u - u^d \nor_{\mathcal U} \nol \grad \left( \rho^{-1} \left( \mathcal P_p \alphap - \alphap^d \right) \right) \nor_{\L^2(\Omega)}.
\end{multline*}
But $\rho^{-1} \left( \mathcal P_p \alphap - \alphap^d \right) \in V_p$, thus $\nol \grad \left( \rho^{-1} \left( \mathcal P_p \alphap - \alphap^d \right) \right) \nor_{\L^2(\Omega)}$ can be estimated by~\eqref{H4},
$$
\nol \grad \left( \rho^{-1} \left( \mathcal P_p \alphap - \alphap^d \right) \right) \nor_{\L^2(\Omega)} \le \frac{C_{1,0}}{\sqrt{\rho_-}} ~ h^{-\theta_{1,0}} ~ \nol \rho^{-\frac{1}{2}} \left( \mathcal P_p \alphap - \alphap^d \right) \nor_{L^2(\Omega)},
$$
where we have used the lower bound $\rho_-$ for $\rho$. This leads to
\begin{multline*}
\frac{1}{2} \frac{\dd}{\dd t} \nol \matl \calPb_q & 0 \\ 0 & \mathcal P_p \matr \matl \Alphaq \\ \alphap \matr - \matl \Alphaq^d \\ \alphap^d \matr \nor_{\mathcal X}^2 
\le \nol \Tens^\frac12 \; \grad \left( \rho^{-1} \left( \alphap - \mathcal P_p \alphap \right) \right) \nor_{\L^2(\Omega)} 
\nol \Tens^\frac12 \; \left( \calPb_q \Alphaq - \Alphaq^d \right) \nor_{\L^2(\Omega)} \\
+ \frac{C_{1,0}}{\sqrt{\rho_-}} ~ h^{-\theta_{1,0}} ~ \nol \Tens \; \left( \Alphaq - \calPb_q \Alphaq \right) \nor_{\L^2(\Omega)}
\nol \rho^{-\frac{1}{2}} \left( \mathcal P_p \alphap - \alphap^d \right) \nor_{L^2(\Omega)} 
+ \frac{C_D}{\sqrt{\rho_-}} \nol u - u^d \nor_{\mathcal U} \nol \rho^{-\frac{1}{2}} \left( \mathcal P_p \alphap - \alphap^d \right) \nor_{L^2(\Omega)} \\
+ \frac{C_D C_{1,0}}{\sqrt{\rho_-}} ~ h^{-\theta_{1,0}} ~ \nol u - u^d \nor_{\mathcal U} \nol \rho^{-\frac{1}{2}} \left( \mathcal P_p \alphap - \alphap^d \right) \nor_{L^2(\Omega)}.
\end{multline*}
Gathering $\nol \Tens^\frac12 \; \left( \calPb_q \Alphaq - \Alphaq^d \right) \nor_{\L^2(\Omega)}$ and $\nol \rho^{-\frac{1}{2}} \left( \mathcal P_p \alphap - \alphap^d \right) \nor_{L^2(\Omega)}$ gives the desired $\mathcal X$-norm:
\begin{multline*}
\frac{1}{2} \frac{\dd}{\dd t} \nol \matl \calPb_q & 0 \\ 0 & \mathcal P_p \matr \matl \Alphaq \\ \alphap \matr - \matl \Alphaq^d \\ \alphap^d \matr \nor_{\mathcal X}^2 
\le \Bigg( \nol \Tens^\frac12 \; \grad \left( \rho^{-1} \left( \alphap - \mathcal P_p \alphap \right) \right) \nor_{\L^2(\Omega)} \\
+ \frac{C_{1,0}}{\sqrt{\rho_-}} ~ h^{-\theta_{1,0}} ~ \nol \Tens \; \left( \Alphaq - \calPb_q \Alphaq \right) \nor_{\L^2(\Omega)} 
+ \frac{C_D}{\sqrt{\rho_-}} \left( 1
+ C_{1,0} ~ h^{-\theta_{1,0}} \right) \nol u - u^d \nor_{\mathcal U} \Bigg)
\nol \matl \calPb_q & 0 \\ 0 & \mathcal P_p \matr \matl \Alphaq \\ \alphap \matr - \matl \Alphaq^d \\ \alphap^d \matr \nor_{\mathcal X}.
\end{multline*}
Dividing both sides by $\nol \matl \calPb_q & 0 \\ 0 & \mathcal P_p \matr \matl \Alphaq \\ \alphap \matr - \matl \Alphaq^d \\ \alphap^d \matr \nor_{\mathcal X}$, we finally get $\dsp \frac{\dd}{\dd t}\mathcal E_2 \le \mathcal E_3$, with
\begin{multline*}
\mathcal E_3 \eqdef \nol \Tens^\frac12 \; \grad \left( \rho^{-1} \left( \alphap - \mathcal P_p \alphap \right) \right) \nor_{\L^2(\Omega)} \\
+ \frac{C_{1,0}}{\sqrt{\rho_-}} ~ h^{-\theta_{1,0}} ~ \nol \Tens \; \left( \Alphaq - \calPb_q \Alphaq \right) \nor_{\L^2(\Omega)} 
+ \frac{C_D}{\sqrt{\rho_-}} \left( 1
+ C_{1,0} ~ h^{-\theta_{1,0}} ~ \right) \nol u - u^d \nor_{\mathcal U}.
\end{multline*}

\item[\textbf{Step 4}]
Using the upper bound $T^+$ for $\Tens$, we get
\begin{multline*}
\mathcal E_3 \eqdef \sqrt{T^+} \nol \grad \left( \rho^{-1} \left( \alphap - \mathcal P_p \alphap \right) \right) \nor_{\L^2(\Omega)} \\
+ \frac{C_{1,0}}{\sqrt{\rho_-}} ~ h^{-\theta_{1,0}} ~ \nol \Tens \; \left( \Alphaq - \calPb_q \Alphaq \right) \nor_{\L^2(\Omega)} 
+ \frac{C_D}{\sqrt{\rho_-}} \left( 1
+ C_{1,0} ~ h^{-\theta_{1,0}} \right) \nol u - u^d \nor_{\mathcal U}.
\end{multline*}
From Lemma~\ref{lem:third-term}, one has
\begin{multline}\label{eq:third-term}
\mathcal E_3 
\le \sqrt{T^+} \left( C_{0,1} C_{1,p} ~ h^{\theta_{1,p}-\theta_{0,1}} 
+ \frac{\rho^+ C_{1,0} C_p}{\sqrt{\rho_-}} ~ h^{\theta_p-\theta_{1,0}} \right) \nol \rho^{-1} \alphap \nor_{H^{\kappa+1}(\Omega)} \\
+ \frac{C_{1,0}}{\sqrt{\rho_-}} ~ h^{-\theta_{1,0}} ~ \nol \Tens \; \left( \Alphaq - \calPb_q \Alphaq \right) \nor_{\L^2(\Omega)} 
+ \frac{C_D}{\sqrt{\rho_-}} \left( 1
+ C_{1,0} ~ h^{-\theta_{1,0}} \right) \nol u - u^d \nor_{\mathcal U}.
\end{multline}
It remains to estimate $\nol \Tens \; \left( \Alphaq - \calPb_q \Alphaq \right) \nor_{\L^2(\Omega)}$. Thanks to the upper bound $T^+$ for $\Tens$ and the third line of~\eqref{eq:relations-between-proj}, we have
$$
\nol \Tens \; \left( \Alphaq - \calPb_q \Alphaq \right) \nor_{\L^2(\Omega)} 
\le \sqrt{T^+} \nol \Tens^{\frac{1}{2}} \; \left( \Alphaq - \Tens^{-1} \Pb_q \Tens \Alphaq \right) \nor_{\L^2(\Omega)}, 
$$
or in other words
$$
\nol \Tens \; \left( \Alphaq - \calPb_q \Alphaq \right) \nor_{\L^2(\Omega)} 
\le \sqrt{T^+} \nol \Tens^{-\frac{1}{2}} \; \left( \Tens \; \Alphaq - \Pb_q \left( \Tens \; \Alphaq \right) \right) \nor_{\L^2(\Omega)}.
$$
With the lower bound $T_-$ for $\Tens$ and~\eqref{H3} with $\v_q = \Tens \; \Alphaq$, this leads to
$$
\nol \Tens \; \left( \Alphaq - \calPb_q \Alphaq \right) \nor_{\L^2(\Omega)} 
\le \frac{\sqrt{T^+}}{\sqrt{T_-}} C_q ~ h^{\theta_q} ~ \nol \Tens \; \Alphaq \nor_{\H^{\kappa+1}(\div ; \Omega)}.
$$
By injecting the latter estimate into~\eqref{eq:third-term}, we get
\begin{multline*}
\mathcal E_3 
\le \sqrt{T^+} \left( C_{0,1} C_{1,p} ~ h^{\theta_{1,p}-\theta_{0,1}} 
+ \frac{\rho^+ C_{1,0} C_p}{\sqrt{\rho_-}} ~ h^{\theta_p-\theta_{1,0}} \right) \nol \rho^{-1} \alphap \nor_{H^{\kappa+1}(\Omega)} \\
+ \frac{\sqrt{T^+} C_{1,0} C_q}{\sqrt{T_- \rho_-}} ~ h^{\theta_q-\theta_{1,0}} ~ \nol \Tens \; \Alphaq \nor_{\H^{\kappa+1}(\div ; \Omega)} 
+ \frac{C_D}{\sqrt{\rho_-}} \left( 1
+ C_{1,0} ~ h^{-\theta_{1,0}} \right) \nol u - u^d \nor_{\mathcal U},
\end{multline*}
which gives, with a rough majoration, the existence of a constant $C_3 > 0$ such that
$$
\mathcal E_3 
\le C_3 ~ h^{\min\{ \theta_{1,p}-\theta_{0,1} \; ; \; \theta_p-\theta_{1,0} \; ; \; \theta_q-\theta_{1,0} \}} ~ \nol \matl \Alphaq \\ \alphap \matr \nor_{\mathcal Z_{\kappa}} 
+ C_3 ~ h^{-\theta_{1,0}} ~ \nol u - u^d \nor_{\mathcal U},
$$
for all $h$ small enough.

\item[\textbf{Step 5}]
By integrating $\dsp \frac{\dd}{\dd t} \mathcal E_2 \le \mathcal E_3$ between $0$ and $t$, the latter inequality gives
\begin{equation}\label{eq:second-term}
\mathcal E_2(t) 
\le \mathcal E_2(0) 
+ C_3 T ~ h^{\min\{ \theta_{1,p}-\theta_{0,1} \; ; \; \theta_p-\theta_{1,0} \; ; \; \theta_q-\theta_{1,0} \}} ~ \nol \matl \Alphaq \\ \alphap \matr \nor_{L^\infty \left([0,T];\mathcal Z_{\kappa} \right)} 
+ C_3 T ~ h^{-\theta_{1,0}} ~ \nol u - u^d \nor_{L^\infty \left([0,T];\mathcal U \right)}.
\end{equation}
Substituting~\eqref{eq:first-term-th} and~\eqref{eq:second-term} into~\eqref{eq:decomposition-error-th}, noticing that $\theta_p \ge \theta_p-\theta_{1,0}$ and $\theta_q \ge \theta_q-\theta_{1,0}$, gives the desired result for all $h$ small enough.
\end{itemize}
\end{proof}

\subsection{The Hamiltonian error}

In this subsection, the numerical analysis focuses on the Hamiltonian, main object of interest in the port-Hamiltonian framework.

The next corollary follows easily from Theorem~\ref{th:General-State}, despite it does not give the expected optimal convergence rate: twice that of the state space absolute error.
\begin{corollary}\label{cor:General-Hamiltonian}
Under the assumptions of Theorem~\ref{th:General-State}, it holds
\begin{equation}\label{eq:General-Hamiltonian}
\val \mathcal E^{\Ham}(t) \var
\le \left( \nol \matl \Alphaq \\ \alphap \matr \nor_{L^\infty([0,T];\mathcal X)} + \frac{\mathcal E^{\mathcal X}(t)}{2} \right) \mathcal E^{\mathcal X}(t), \Forall t \in [0,T].
\end{equation}
\end{corollary}

\begin{proof}
It is straightforward that
$$
\val \mathcal E^\Ham \var = \frac12 \val \psl \matl \Alphaq \\ \alphap \matr + \matl \Alphaq^d \\ \alphap^d \matr, \matl \Alphaq \\ \alphap \matr - \matl \Alphaq^d \\ \alphap^d \matr \psr_{\mathcal X} \var
\le \frac12 \nol \matl \Alphaq \\ \alphap \matr + \matl \Alphaq^d \\ \alphap^d \matr \nor_{\mathcal X} \mathcal E^{\mathcal X}.
$$
But
$$
\nol \matl \Alphaq(t) \\ \alphap(t) \matr + \matl \Alphaq^d(t) \\ \alphap^d(t) \matr \nor_{\mathcal X} 
\le 2 \nol \matl \Alphaq \\ \alphap \matr \nor_{L^\infty([0,T];\mathcal X)} 
+ \mathcal E^{\mathcal X}(t),
\Forall t \in [0,T],
$$
which ends the proof of~\eqref{eq:General-Hamiltonian}.
\end{proof}

In the following theorem, it is proved that \emph{compatibility conditions} between $\H_q$, $H_p$ and $H_\partial$, including but not restricted to those leading to the preservation of the de Rham cohomology at the discrete level, lead to a far better result for $\mathcal E^{\Ham}$.
\begin{theorem}\label{th:General-Hamiltonian}
Under the assumptions of Theorem~\ref{th:General-State}, assume furthermore that
\begin{itemize}
\item $\psl \grad \left( \frac{v_p - \mathcal P_p v_p}{\rho} \right), \v_q^d \psr_{\L^2(\Omega)} = 0$, for all $\v_q^d \in \H_q$, $v_p \in H^1(\Omega)$;
\item $\psl \Tens \; \left( \v_q - \calPb_q \v_q \right), \grad \left( v_p^d \right) \psr_{\L^2(\Omega)} = 0$, for all $v_p^d \in H_p$, $\v_q \in \L^2(\Omega)$;
\item $\psl u - u^d, \gamma_0 \left( v_p^d \right) \psr_{L^2(\partial\Omega)} = 0$, for all $u \in L^2(\partial\Omega)$, $u^d$ approximation of $u$ in $H_\partial$, $v_p^d \in H_p$;
\item $\psl u^d, \gamma_0 \left( \frac{v_p^d - \mathcal P_p v_p}{\rho} \right) \psr_{L^2(\partial\Omega)} = 0$, for all $u^d \in H_\partial$, $v_p \in H^1(\Omega)$, $v_p^d$ approximation of $v_p$ in $H_p$.
\end{itemize}
Then
$$
\mathcal E^{\Ham}(t) - \mathcal E^{\Ham}(0) = \frac{1}{2} \left( \left( \mathcal E^{\mathcal X}(t) \right)^2 - \left( \mathcal E^{\mathcal X}(0) \right)^2 \right), \Forall t \in [0,T].
$$
\end{theorem}

\begin{proof}
Obviously
$$
\left( \mathcal E^{\mathcal X} \right)^2 
= \nol \matl \Alphaq \\ \alphap \matr \nor_{\mathcal X}^2
	- 2 \psl \matl \Alphaq \\ \alphap \matr, \matl \Alphaq^d \\ \alphap^d \matr \psr_{\mathcal X}
	+ \nol \matl \Alphaq^d \\ \alphap^d \matr \nor_{\mathcal X}^2.
$$
Hence
\begin{equation}\label{eq:Ham-in-terms-of-others}
\mathcal E^{\Ham} 
= \frac{1}{2} \left( \mathcal E^{\mathcal X} \right)^2 
	+ \psl \matl \Alphaq \\ \alphap \matr, \matl \Alphaq^d \\ \alphap^d \matr \psr_{\mathcal X}
	- \nol \matl \Alphaq^d \\ \alphap^d \matr \nor_{\mathcal X}^2.
\end{equation}
From~\eqref{eq:FV-continuous} and~\eqref{eq:FV-discrete} together with Lemma~\ref{lem:energy--co-energy}, it is straightforward that
\begin{multline*}
\frac{\dd}{\dd t} \psl \matl \Alphaq \\ \alphap \matr, \matl \Alphaq^d \\ \alphap^d \matr \psr_{\mathcal X} 
= \psl \grad \left( \rho^{-1} \alphap - \frac{\mathcal P_p \alphap}{\rho} \right), \Tens \; \Alphaq^d \psr_{\L^2(\Omega)} \\
	+ \psl \Tens \; \calPb_q \Alphaq - \Tens \; \Alphaq, \grad \left( \frac{\alphap^d}{\rho} \right) \psr_{\L^2(\Omega)} 
	+ \psl u, \gamma_0 \left( \frac{\alphap^d}{\rho} \right) \psr_{L^2(\partial\Omega)} 
	+ \psl u^d, \gamma_0 \left( \frac{\mathcal P_p \alphap}{\rho} \right) \psr_{L^2(\partial\Omega)},
\end{multline*}
which becomes, thanks to the assumptions on the $q$- and $p$-type families (remember that $\Alphaq^d \in \V_q \eqdef \Tens^{-1} \; \H_q$ and $\alphap^d \in V_p \eqdef \rho H_p$)
$$
\frac{\dd}{\dd t} \psl \matl \Alphaq \\ \alphap \matr, \matl \Alphaq^d \\ \alphap^d \matr \psr_{\mathcal X} 
= \psl u, \gamma_0 \left( \frac{\alphap^d}{\rho} \right) \psr_{L^2(\partial\Omega)} 
+ \psl u^d, \gamma_0 \left( \frac{\mathcal P_p \alphap}{\rho} \right) \psr_{L^2(\partial\Omega)}.
$$
Since by Proposition~\ref{prop:discrete-power-balance}
$$
\frac{\dd}{\dd t} \nol \matl \Alphaq^d \\ \alphap^d \matr \nor_{\mathcal X}^2 
= 2 \psl u^d, y^d \psr_{L^2(\partial\Omega)},
$$
it can be deduced that
$$
\frac{\dd}{\dd t} \left( \psl \matl \Alphaq \\ \alphap \matr, \matl \Alphaq^d \\ \alphap^d \matr \psr_{\mathcal X}
	- \nol \matl \Alphaq^d \\ \alphap^d \matr \nor_{\mathcal X}^2 \right) 
= \psl u - u^d, \gamma_0 \left( \frac{\alphap^d}{\rho} \right) \psr_{L^2(\partial\Omega)} 
	+ \psl u^d, \gamma_0 \left( \frac{\mathcal P_p \alphap - \alphap^d}{\rho} \right) \psr_{L^2(\partial\Omega)}.
$$
Now, thanks to the assumptions involving the boundary finite element families, an integration in time from $0$ to $t$ gives
$$
\psl \matl \Alphaq \\ \alphap \matr, \matl \Alphaq^d \\ \alphap^d \matr \psr_{\mathcal X}
- \nol \matl \Alphaq^d \\ \alphap^d \matr \nor_{\mathcal X}^2
= \psl \matl {\Alphaq}_0 \\ {\alphap}_0 \matr, \matl \Alphaq^d(0) \\ \alphap^d(0) \matr \psr_{\mathcal X}
- \nol \matl \Alphaq^d(0) \\ \alphap^d(0) \matr \nor_{\mathcal X}^2,
$$
and the result follows from~\eqref{eq:Ham-in-terms-of-others} by subtracting $\mathcal E^{\Ham}(0)$ from both side.
\end{proof}

\begin{remark}
Note that the anisotropy and heterogeneity have a non-negligible influence for the validity of Theorem~\ref{th:General-Hamiltonian}, as these induce \emph{curly} projectors $\calPb_q$ and $\mathcal P_p$. These modifications of the compatibility conditions for this more general case would be hidden if the analysis was carried out with constant parameters.
\end{remark}

\subsection{Other causality: the Dirichlet boundary control}
\label{Sec:Switch}

Theorem~\ref{th:General-State} has its counterpart for the other causality, already  discussed in Section~\ref{Sec:switched-system}, where $u$ and $y$ have been switched (notation S). For the sake of space saving, since the proof is quite similar, only the result for the general framework are briefly stated below.

Assuming an existence and regularity result such as~\eqref{H0} for the case~\eqref{eq:waves-continuous}--\eqref{eq:waves-boundary-switch} (though not covered in~\cite{KurZwa15}), one can easily adapt the proof of Theorem~\ref{th:General-State}.

If
\begin{itemize}
\item $\H_q$ is $\H(\div ; \Omega)$-conforming (instead of $\L^2(\Omega)$-conforming);
\item $H_p$ is $L^2(\Omega)$-conforming (instead of $H^1(\Omega)$-conforming);
\item $H_\partial$ is $H^\frac{1}{2}(\partial\Omega)$-conforming,
\end{itemize}
satisfying
\begin{equation}\label{H1S}\tag{\ref{H1}S}
\exists C_p > 0, \; \exists \theta_p \ge 0 \; : \; \nol P_p v_p - v_p \nor_{L^2(\Omega)} \le C_p ~ h^{\theta_p} ~ \nol v_p \nor_{H^{\kappa+1}(\Omega)}, 
\Forall v_p \in H^{\kappa+1}(\Omega), \; \forall h \in (0,h^*);
\end{equation}
\begin{multline}\label{H2S}\tag{\ref{H2}S}
\exists C_{\divs,p} > 0, \; \exists \theta_{\divs,p} \ge 0 \; : \; \nol \Pb_{\divs,q} \v_q - \v_q \nor_{\H(\div ; \Omega)} \le C_{\divs,p} ~ h^{\theta_{\divs,p}} ~ \nol \v_p \nor_{\H^{\kappa+1}(\div ; \Omega)}, \\
\Forall \v_q \in \H^{\kappa+1}(\div ; \Omega), \; \forall h \in (0,h^*);
\end{multline}
where $\Pb_{\divs,q}$ is the $\H(\div ; \Omega)$-orthogonal projector from $\H(\div ; \Omega)$ onto $\V_q$;
\begin{equation}\label{H3S}\tag{\ref{H3}S}
\exists C_q > 0, \; \exists \theta_q \ge 0 \; : \; \nol \Pb_q \v_q - \v_q \nor_{\L^2(\Omega)} \le C_q ~ h^{\theta_q} ~ \nol \v_q \nor_{\H^{\kappa+1}(\div ; \Omega)}, 
\Forall \v_q \in \H^{\kappa+1}(\div ; \Omega), \; \forall h \in (0,h^*);
\end{equation}
\begin{multline}\label{H4S}\tag{\ref{H4}S}
\exists C_{\divs,0} > 0, \; \exists \theta_{\divs,0} \ge 0 \; : \; \nol \div \left( \v_q^d \right) \nor_{L^2(\Omega)} \le C_{\div\rightarrow0} ~ h^{-\theta_{\divs,0}} ~ \nol \v_q^d \nor_{\L^2(\Omega)}, \\
\Forall \v_q^d \in \H_q, \; \forall h \in (0,h^*);
\end{multline}
\begin{multline}\label{H5S}\tag{\ref{H5}S}
\exists C_{0,\divs} > 0, \; \exists \theta_{0,\divs} \ge 0 \; : \; \nol \Pb_q \v_q - \v_q \nor_{\H(\div ; \Omega)} \le C_{0,\divs} ~ h^{-\theta_{0,\divs}} ~ \nol \Pb_{\divs,q} \v_q - \v_q \nor_{\H(\div ; \Omega)}, \\
\Forall \v_q \in \H(\div ; \Omega), \; \forall h \in (0,h^*);
\end{multline}
we have the following theorem.
\begin{theorem}\label{th:General-State-Switch}
Let $\kappa>0$ be an integer and $\Omega$ be of class $\mathcal C^{\kappa+2}$. There exists a constant $C_{\mathcal X}>0$ such that for all $T>0$, all initial data $\matl {\Alphaq}_0 \\ {\alphap}_0 \matr \in \mathcal Z_{\kappa}$, all $u \in \mathcal C^2([0,\infty);\mathcal U_\kappa)$ such that $u(0) = \gamma_\perp \left( \Tens \; {\Alphaq}_0 \right)$, and all $h\in (0,h^*)$
\begin{multline*}
\widetilde{\mathcal E}^{\mathcal X}(t)
\le \nol \matl \calPb_q & 0 \\ 0 & \mathcal P_p \matr \matl {\Alphaq}_0 \\ {\alphap}_0 \matr - \matl \Alphaq^d(0) \\ \alphap^d(0) \matr \nor_{\mathcal X} 
+ \widetilde C_{\mathcal X} \max \{ 1, T \} ~ h^{\widetilde \theta^*} ~ \nol \matl \Alphaq \\ \alphap \matr \nor_{L^\infty([0,T];\mathcal Z_{\kappa})} \\
+ \widetilde C_{\mathcal X} T ~ h^{-\theta_{\divs,0}} ~ \nol u - u^d \nor_{L^\infty([0,T];\mathcal U)}, \Forall t \in [0,T],
\end{multline*}
where
$$
\widetilde \theta^* \eqdef \min \left\{ \theta_{\divs,q}-\theta_{0,\divs} \, ; \, \theta_q-\theta_{\divs,0} \, ; \, \theta_q-\theta_{\divs,0} \right\}.
$$
\end{theorem}

\begin{remark}
A counterpart of Theorem~\ref{th:General-Hamiltonian} should also be possible to prove under suitable additional compatibility conditions.
\end{remark}

\section{Optimal orders of conforming finite elements}
\label{Sec:AnaNum-PkRTl}

The purpose of this section is to provide optimal combinations of usual finite elements minimizing the number of degrees of freedom for a given convergence rate, illustrating the abstract estimates obtained in Theorems~\ref{th:General-State} and~\ref{th:General-Hamiltonian}. The errors to be analysed when the mesh size parameter $h>0$ tends towards $0$ are those of the previous section, and will be numerically investigated in the next section.

\subsection{Mesh assumptions}

These classical assumptions in numerical analysis for usual finite elements (see \eg \cite{BofBreFor13,Gat14}) are recalled for the sake of completeness.

The mesh family $(\Tau_h)_{h\in(0,h^*)}$ of $\Omega$ will be supposed to be a collection of simplicial, regular and quasi-uniform triangularization of $\overline{\Omega}$, meaning that
\begin{enumerate}
\item
It is given by a collection of triangles or tetrahedra, denoted $K$ in the sequel.
\item
If $h_K>0$ denotes the diameter of $K$, \ie $h_K \eqdef \max_{\x,\y \in K} \val \x - \y \var$, and $d_K$ is the diameter of the inscribed circle or sphere in $K$, there exists a constant $C>0$, independent of $h$ such that
$$
\frac{h_K}{d_K} \le C, \Forall K \in \Tau_h, \Forall h\in(0,h^*).
$$
The mesh parameter is then defined as $h \eqdef \max_{K \in \Tau_h} h_K$.
\item
There exists a constant $c>0$ independent of $h$ such that $\min_{K \in \Tau_h} h_K \ge c h$ for all $h\in(0,h^*)$.
\end{enumerate}

\subsection{Lagrange element}
\label{Sec:Lagrange}

The three finite element families must be $\L^2(\Omega)$-, $H^1(\Omega)$- and $H^\frac{1}{2}(\partial\Omega)$-conforming respectively, in order to apply Theorem~\ref{th:General-State}. A first easy choice is to take the usual continuous Galerkin finite elements for families $q$ and $p$, \ie
$$
\H_q \eqdef \left\{ \v_q^d \in \left(C(\overline{\Omega})\right)^N \; \mid \; \v_q^d\Big|_K \in \left(\mathbb P_k(K)\right)^N, \; \forall K \in \Tau_h \right\},
$$
when $k\ge1$, or the piecewise constant functions when $k=0$.
$$
H_p \eqdef \left\{ v_p^d \in C(\overline{\Omega}) \; \mid \; v_p^d\Big|_K \in \mathbb P_\ell(K), \; \forall K \in \Tau_h \right\},
$$
where $\ell\ge1$.

In the above definition, $\mathbb P_j(K)$ is the Lagrange finite element of order $j$ made of all polynomials of degree less or equal to $j$ on $K$.

The space $H_p$ is known as continuous Galerkin of order $\ell$ as we impose continuity of basis functions. The space $\H_q$ is the vectorial counterpart of $H_p$, at order $k$. In the sequel, we will refer to these spaces \via $CG_\ell$ and $CG_k$ respectively.

For the discretization space at the boundary $H_\partial$, discontinuous Galerkin finite elements of order $m\ge0$, denoted $DG_m$, are chosen
$$
H_\partial \eqdef \left\{ v_\partial^d \in L^\infty(\partial{\Omega}) \; \mid \; v_\partial^d\Big|_E \in \mathbb P_m(E), \; \forall E, \text{ edges of } K \in \Tau_h \text{ located at the boundary} \right\}.
$$
They are indeed in $H^\frac{1}{2}(\partial\Omega)$ for all $m \ge 0$. Practically, it consists in taking the Dirichlet trace of discontinuous Galerkin of order $m$ on the whole domain, \ie keeping only degrees of freedom located at the boundary.

All error estimates are well-known and can be found \eg in \cite{BofBreFor13,Gat14} and references therein for the global interpolation operators. Obviously, these estimates hold true for the orthogonal projectors. Then~\eqref{H1}--\eqref{H2}--\eqref{H3}--\eqref{H4}--\eqref{H5} read
\begin{equation}\label{H1L}\tag{\ref{H1}L}
\exists C_p > 0, \quad \nol P_p v_p - v_p \nor_{L^2(\Omega)} \le C_p ~ h^{\ell+1} ~ \nol v_p \nor_{H^{\ell+1}(\Omega)},
\Forall v_p \in H^{\ell+1}(\Omega), \; \forall h \in (0,h^*);
\end{equation}
\begin{equation}\label{H2L}\tag{\ref{H2}L}
\exists C_{1,p} > 0, \quad \nol P_{1,p} v_p - v_p \nor_{H^1(\Omega)} \le C_{1,p} ~ h^{\ell} ~ \nol v_p \nor_{H^{\ell+1}(\Omega)},
\Forall v_p \in H^{\ell+1}(\Omega), \; \forall h \in (0,h^*).
\end{equation}
The following estimate requires more attention. Nevertheless, a careful analysis using~\eqref{H1L}, \cite[Proposition~1.4]{FoiTem78}, and a density argument, leads to
\begin{equation}\label{H3L}\tag{\ref{H3}L}
\exists C_q > 0, \quad \nol \Pb_q \v_q - \v_q \nor_{L^2(\Omega)} \le C_q ~ h^{k+1} ~ \nol \v_q \nor_{\H^{k+1}(\div ; \Omega)}, 
\Forall \v_q \in \H^{k+1}(\div ; \Omega), \; \forall h \in (0,h^*).
\end{equation}
\begin{equation}\label{H4L}\tag{\ref{H4}L}
\exists C_{1,0} > 0, \quad \nol \grad \left( v_p^d \right) \nor_{\L^2(\Omega)} \le C_{1,0} ~ h^{-1} ~ \nol v_p^d \nor_{L^2(\Omega)}, \Forall v_p^d \in H_p, \; \forall h \in (0,h^*),
\end{equation}
and finally
\begin{equation}\label{H5L}\tag{\ref{H5}L}
\exists C_{0,1} > 0, \quad \nol P_p v_p - v_p \nor_{H^1(\Omega)} \le C_{0,1} ~ \nol P_{1,p} v_p - v_p \nor_{H^1(\Omega)}, 
\Forall v_p \in H^1(\Omega), \; \forall h \in (0,h^*).
\end{equation}

Remark that, with the choice of $H_\partial$, a straightforward but tedious exercise, using lifting operators, Bramble-Hilbert Theorem~\cite{BofBreFor13}, continuity of trace operators and quasi-uniform hypothesis give that
\begin{equation}\label{eq:Galerkin-boundary}
\nol u - u^d \nor_{H^{-\frac{1}{2}}(\partial\Omega)} \le C_u ~ h^{m+1} ~ \nol u \nor_{H^{m-\frac{1}{2}}(\partial\Omega)}, \Forall u \in H^{m-\frac{1}{2}}(\partial\Omega),
\end{equation}
at the boundary.

The following holds true.
\begin{theorem}\label{th:Pk-Pl-State}
Let $\kappa>0$ be an integer, $T>0$, $\matl {\Alphaq}_0 \\ {\alphap}_0 \matr \in \mathcal Z_{\kappa}$, $u \in \mathcal C^2([0,\infty);H^{\kappa-\frac{1}{2}}(\partial\Omega))$, and $\matl \Alphaq^d(0) \\ \alphap^d(0) \matr$ and $u^d$ their respective continuous and discontinuous Galerkin interpolations in $\V_q \times V_p$ and $H_\partial$ given by the finite elements $\left(CG_{k}\right)^N \times CG_{\ell} \times DG_{m}$.

There exists a constant $C>0$, independent of $T>0$, $\matl {\Alphaq}_0 \\ {\alphap}_0 \matr$, and $u$, such that for all $h$ small enough and all $t \in [0,T]$
\begin{equation}\label{eq:th-General-State-Pk-Pl}
\mathcal E^{\mathcal X}(t)
\le C \max \{ 1, T \} ~ h^{\min \left\{ \ell \, ; \, k \, ; \, m \right\}} ~ \left( \nol \matl \Alphaq \\ \alphap \matr \nor_{L^\infty([0,T];\mathcal Z_{\kappa})} + \nol u \nor_{L^\infty([0,T];H^{\kappa-\frac{1}{2}}(\partial\Omega))} \right).
\end{equation}
Furthermore, the optimal order is $\kappa$, obtained with $k=\kappa$, $\ell=\kappa$ and $m=\kappa-1$.
\end{theorem}

\begin{proof}
Since $\kappa\ge1$, $H^{\kappa-\frac{1}{2}}(\partial\Omega) = \mathcal{U}_\kappa
 \subset H^{\frac{1}{2}}(\partial\Omega)$, and $u$ can indeed be approximated in $H_\partial$.
 
From~\eqref{H0}, the solution to~\eqref{eq:waves-continuous}-\eqref{eq:waves-boundary} belongs to $\mathcal Z_{\kappa}$ continuously in time. Recall that this means
$$
\Alphaq \in \Tens^{-1} \H^{\kappa+1}(\div ; \Omega) \eqdef \left\{ \v \in \L^2(\Omega) \: \mid \: \Tens \; \v \in \H^{\kappa}(\Omega), \: \div\left(\Tens \; \v\right) \in H^{\kappa}(\Omega) \right\},
$$
and $\alphap \in \rho H^{\kappa+1}(\Omega)$. Hence, following~\eqref{H1L}--\eqref{H2L}--\eqref{H3L}--\eqref{eq:Galerkin-boundary}, the order of the finite element families satisfies
$$
\theta_p = \ell + 1 \le \kappa + 1, \quad 
\theta_{1,p} = \ell \le \kappa, \quad 
\theta_q = k + 1 \le \kappa + 1, \quad 
\theta_u := m + 1 \le \kappa.
$$
From~\eqref{H4L}--\eqref{H5L}
$$
\theta_{1,0} = 1, \quad 
\theta_{0,1} = 0.
$$
One deduces the convergence rate of $\mathcal E^{\mathcal X}$ thanks to~\eqref{eq:th-General-State} and~\eqref{eq:rate-general}, \ie it is given by
$$
\theta^* = \min \left\{ \ell \, ; \, k \, ; \, m+1 \right\},
$$
where we have used~\eqref{H1L} and~\eqref{H3L} for the approximation of the initial data. Now, taking into account the maximal regularities given by~\eqref{H0} (and the assumed regularity on $u$) leads to the maximal rate $\min \left\{ \kappa \, ; \, \kappa \, ; \, \kappa  \right\} = \kappa$.

Finally, one gets this maximal order with the minimal number of degrees of freedom when we take $\ell = k = m+1 = \kappa$.
\end{proof}

\subsection{Other finite element families}
\label{Sec:other-choices}

Following \cite[Proposition~2.5.4.]{BofBreFor13}, estimate~\eqref{H3}, and thus Theorem~\ref{th:General-State}, hold true for many usual $\H(\div ; \Omega)$-conforming families (hence $\L^2(\Omega)$-conforming as required, or even with curl-conforming finite element), namely: Raviart-Thomas $RT_k$ (for an introduction to this important class of finite element, see \eg \cite{Gat14}), Brezzi-Douglas-Marini $BDM_k$ and discontinuous Galerkin finite elements $DG_k$.

\begin{proposition}\label{prop:other-choices}
Let $\kappa>0$ be an integer, $T>0$, $\matl {\Alphaq}_0 \\ {\alphap}_0 \matr \in \mathcal Z_{\kappa}$, $u \in \mathcal C^2([0,\infty);H^{\kappa-\frac{1}{2}}(\partial\Omega))$, and $\matl \Alphaq^d(0) \\ \alphap^d(0) \matr$ and $u^d$ their respective interpolations. The optimal rate of convergence is reached with $\V_q \times V_p \times H_\partial$ given by
$$
DG_{\kappa-1} \times CG_\kappa \times DG_{\kappa-1},
\qquad
RT_\kappa \times CG_\kappa \times DG_{\kappa-1},
\qquad
BDM_\kappa \times CG_\kappa \times DG_{\kappa-1},
$$
$$
CG_{\kappa} \times CG_\kappa \times CG_{\kappa},
\qquad
DG_{\kappa-1} \times CG_\kappa \times CG_{\kappa},
\qquad
RT_\kappa \times CG_\kappa \times CG_{\kappa},
\qquad
BDM_\kappa \times CG_\kappa \times CG_{\kappa},
$$
all of them leading to the same convergence rate $\kappa$.
\end{proposition}

\begin{proof}
This is a direct application of Theorem~\ref{th:General-State}.
\end{proof}

\begin{remark}
Care must be taken with the subscript of $RT$ element, which may differ from one source to another, depending on how the lowest order is denoted: either $RT_0$ or $RT_1$. In this paper, we stick to the definition given in  FEniCS~\cite{FEniCS}, the software being used for the simulations ran in Section~\ref{Sec:SimuNum}, and denote the lowest order by $RT_1$.
\end{remark}

\section{Numerical study of the convergence rate in 2D}
\label{Sec:SimuNum}

In this section, simulations are performed to illustrate our results. More precisely, we intend to verify if the convergence rates are indeed those proved in Theorem~\ref{th:Pk-Pl-State} and claimed in Section~\ref{Sec:other-choices}.

\subsection{An analytical solution}

In order to study the convergence rate, we propose to focus on a 2D toy model, isotropic and heterogeneous, for which an analytical solution is known. This choice is made to avoid the computation of a reference solution.

Let us consider $\Omega = (0,1) \times (0,1)$. The physical parameters are $\rho \equiv 1$ and $\Tens \equiv \overline{\overline{I}}$. Denoting $f(t) \eqdef 2\sin\left(\sqrt{2}t\right) + 3\cos\left(\sqrt{2}t\right)$, we define
$$
\Alphaq \eqdef f(t) \matl -\sin(x)\sin(y) \\ \cos(x)\cos(y) \matr, \qquad \alphap \eqdef \frac{\dd}{\dd t} f(t) \cos(x)\sin(y), \Forall (x,y) \in \Omega, \: t \ge 0,
$$
and
$$
u(t) \eqdef \left\{
\begin{array}{ll}
- f(t) \cos(x), & \Forall (x,y) \in (0,1)\times\{0\}, \\
- f(t) \sin(1) \sin(y), & \Forall (x,y) \in \{1\}\times(0,1), \\
f(t) \cos(x) \cos(1), & \Forall (x,y) \in (0,1)\times\{1\}, \\
0, & \Forall (x,y) \in \{0\}\times(0,1).
\end{array}\right.
$$
Then, $\matl \Alphaq \\ \alphap \matr$ is a $\mathcal C^\infty([0,\infty);\mathcal C^\infty(\Omega))$-solution to the wave equation written as a pHs~\eqref{eq:Hamilton-System}--\eqref{eq:Hamilton-System-Boundary}.

The choice of sine and cosine functions has been made to avoid exact interpolation in the polynomial finite element spaces of high order.

The Hamiltonian is easily obtained for all $t\ge0$
\begin{multline*}
\Ham\left(\Alphaq(t),\alphap(t)\right) 
= \frac{1}{8} \left( \frac{\dd}{\dd t} f(t) \right)^2 \left( 1 - (\sin(1)\cos(1))^2 \right) \\
+ \frac{1}{8} \left( f(t) \right)^2
\times \Big\{ \left( 1 + \sin(1)\cos(1) \right)^2 + \left( 1 - \sin(1)\cos(1) \right)^2 \Big\}.
\end{multline*}

\subsection{Simulations}

In this section, the following procedure is proposed: $\H_q$, $H_p$ and $H_\partial$ are varying according to many ranges of finite element families, and all combinations are tested to analyse the behavior of the convergence rate.

The simulations are performed using FEniCS~\cite{FEniCS}, with a Crank-Nicolson scheme in time $t\in(0,0.5)$~\cite{PETScTS}. The time step is chosen small enough according to $\kappa$, in order to ensure that the error is driven by the mesh size $h$ and not the time step\footnote{This is not a kind of CFL condition due to the PFEM, but a matter of \emph{space} discretization analysis. We only have access to the total error in $O(h^\kappa) + O(\Delta t^2)$. Taking a sufficiently small time step allows for the analysis of this error as a function of $h$, to be able to observe the order $\kappa$. Note in particular that Crank-Nicolson scheme is unconditionality stable and symplectic.}  $dt=10^{-5}$. These tests 
have been run on a personal computer (Intel Core I7 processor, 24GB of RAM).

\begin{table}[ht!]
\centering
\includegraphics[width=0.8\textwidth]{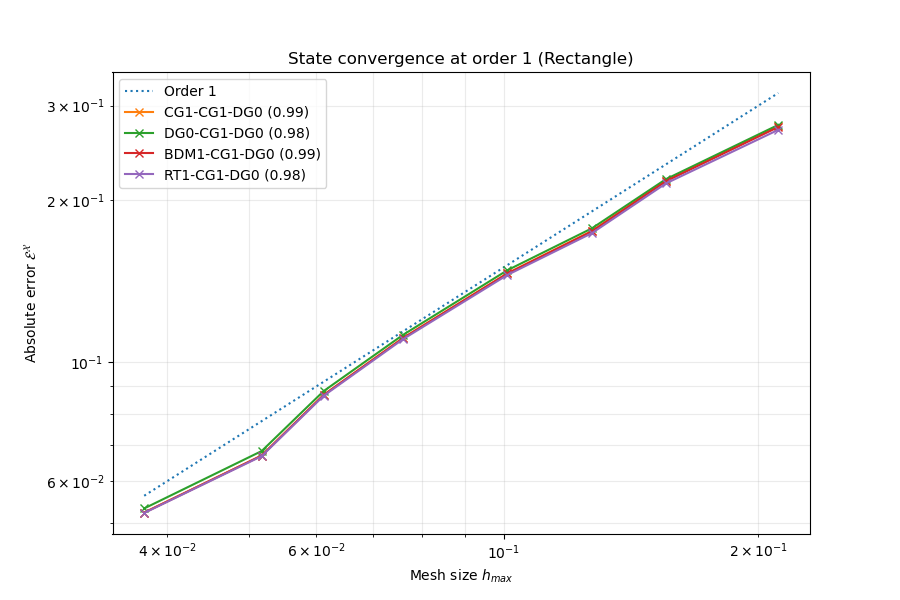}

\vspace{2em}
\begin{tabular}{c|ccccccccccccc}
\hline 
$DG_0$ & $DG_0$	& $DG_1$ & $DG_2$ & $DG_3$ & $CG_1$ & $CG_2$ & $CG_3$ & $BDM_1$ & $BDM_2$ & $BDM_3$ & $RT_1$ & $RT_2$ & $RT_3$\\
\hline
\hline
$DG_\ell$ & -0.5 	& -0.5 	& -0.5 	& -0.5 	& -0.5 	& -0.5 	& -0.5 	& -0.5 	& -0.5 	& -0.5 	& -0.5 	& -0.5 	& -0.5 \\
\hline
$CG_1$ & \textbf{0.98} 	& 0.99 	& 0.98 	& 0.98 	& \textbf{0.99} 	& 0.99 	& 0.98 	& \textbf{0.99} 	& 0.98 	& 0.98 	& \textbf{0.98} 	& 0.98 	& 0.98 \\
\hline
$CG_2$ & -0.0 	& 0.99 	& 0.99 	& 0.99 	& 1.0 	& 0.99 	& 0.99 	& 1.0 	& 0.99 	& 0.99 	& 0.99 	& 0.99 	& 0.99 \\
\hline
$CG_3$ & -0.0 	& 1.01 	& 0.99 	& 0.99 	& 0.53 	& 0.45 	& 0.99 	& 0.98 	& 0.99 	& 0.99 	& 0.98 	& 0.99 	& 0.99 \\
\hline
\end{tabular}
\begin{center}
Optimal order given by Theorem~\ref{th:Pk-Pl-State} and Proposition~\ref{prop:other-choices}: $\theta^* = 1$.
\end{center}

\vspace{2em}
\caption{Convergence at order $1$ for $\mathcal E^{\mathcal X}$ obtained for different combinations of finite element families. The first cell (the upper-left one) gives the type of {\it boundary finite element}: $DG_0$, columns correspond to the $q$-type variables, and rows to the $p$-type ones. The order in boldface are the optimal order given by Theorem~\ref{th:Pk-Pl-State} and Proposition~\ref{prop:other-choices}.}
\label{tab:EX1}
\end{table}

\begin{table}[ht!]
\centering
\includegraphics[width=0.8\textwidth]{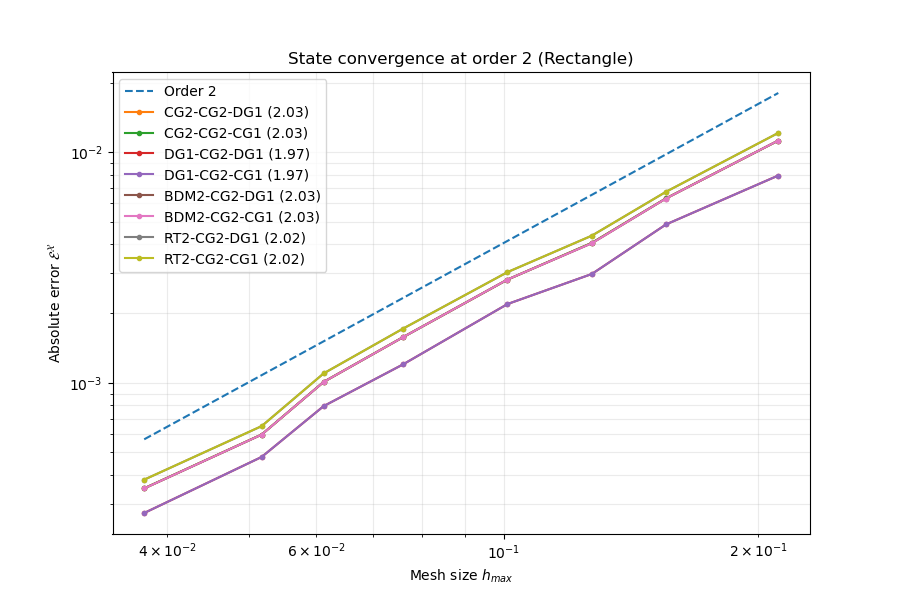}

\vspace{2em}
\begin{tabular}{c|ccccccccccccc}
\hline 
$DG_1$ & $DG_0$	& $DG_1$ & $DG_2$ & $DG_3$ & $CG_1$ & $CG_2$ & $CG_3$ & $BDM_1$ & $BDM_2$ & $BDM_3$ & $RT_1$ & $RT_2$ & $RT_3$\\
\hline
\hline
$DG_\ell$ & -0.5 	& -0.5 	& -0.5 	& -0.5 	& -0.5 	& -0.5 	& -0.5 	& -0.5 	& -0.5 	& -0.5 	& -0.5 	& -0.5 	& -0.5 \\
\hline
$CG_1$ & 0.99 	& 0.99 	& 0.97 	& 0.95 	& 1.75 	& 1.45 	& 0.96 	& 1.51 	& 1.04 	& 0.96 	& 1.01 	& 1.02 	& 0.96 \\
\hline
$CG_2$ & 0.0 	& \textbf{1.97} 	& 2.03 	& 2.03 	& 1.45 	& \textbf{2.03} 	& 2.03 	& 1.49 	& \textbf{2.03} 	& 2.03 	& 1.02 	& \textbf{2.02} 	& 2.03 \\
\hline
$CG_3$ & 0.02 	& 1.5 	& 2.03 	& 2.03 	& 0.94 	& 2.03 	& 2.03 	& 0.97 	& 2.02 	& 2.03 	& 0.57 	& 2.02 	& 2.03 \\
\hline
\end{tabular}

\vspace{2em}
\begin{tabular}{c|ccccccccccccc}
\hline 
$CG_1$ & $DG_0$	& $DG_1$ & $DG_2$ & $DG_3$ & $CG_1$ & $CG_2$ & $CG_3$ & $BDM_1$ & $BDM_2$ & $BDM_3$ & $RT_1$ & $RT_2$ & $RT_3$\\
\hline
\hline
$DG_\ell$ & -0.5 	& -0.5 	& -0.5 	& -0.5 	& -0.5 	& -0.5 	& -0.5 	& -0.5 	& -0.5 	& -0.5 	& -0.5 	& -0.5 	& -0.5 \\
\hline
$CG_1$ & 0.99 	& 0.99 	& 0.97 	& 0.95 	& 1.75 	& 1.45 	& 0.96 	& 1.51 	& 1.04 	& 0.96 	& 1.01 	& 1.02 	& 0.96 \\
\hline
$CG_2$ & 0.0 	& \textbf{1.97} 	& 2.03 	& 2.03 	& 1.45 	& \textbf{2.03} 	& 2.03 	& 1.49 	& \textbf{2.03} 	& 2.03 	& 1.02 	& \textbf{2.02} 	& 2.03 \\
\hline
$CG_3$ & 0.02 	& 1.5 	& 2.03 	& 2.03 	& 0.94 	& 2.03 	& 2.03 	& 0.97 	& 2.02 	& 2.03 	& 0.57 	& 2.02 	& 2.03 \\
\hline
\end{tabular}

\vspace{2em}
\caption{Convergence at order $2$ for $\mathcal E^{\mathcal X}$ obtained for different combinations of finite element families. The first cell (the upper-left one) of each table gives the type of {\it boundary finite element}: $DG_1$ or $CG_1$, columns correspond to the $q$-type variables, and rows to the $p$-type ones. The order in boldface are the optimal order given by Theorem~\ref{th:Pk-Pl-State} and Proposition~\ref{prop:other-choices}.}
\label{tab:EX2}
\end{table}

\begin{table}[ht!]
\centering
\includegraphics[width=0.8\textwidth]{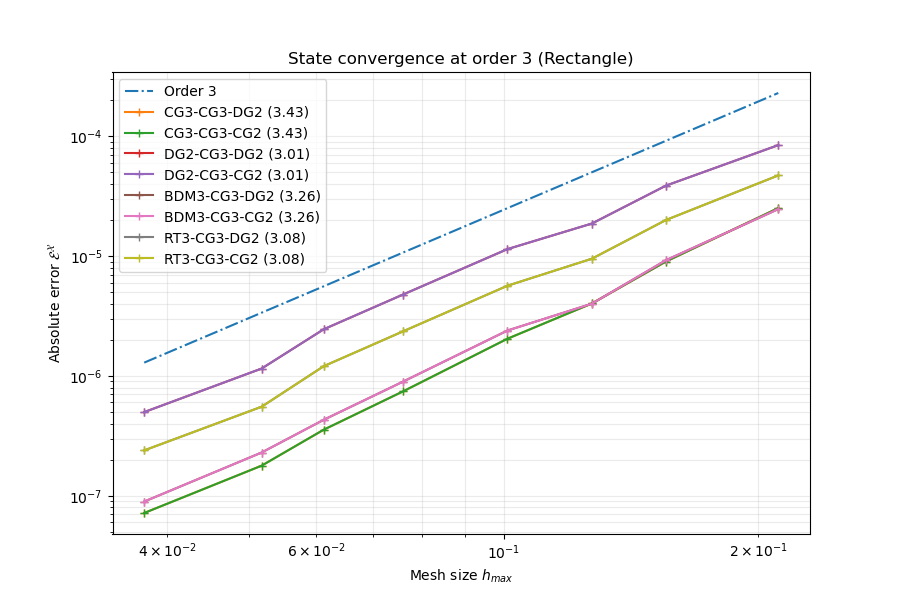}

\vspace{2em}
\begin{tabular}{c|ccccccccccccc}
\hline 
$DG_2$ & $DG_0$	& $DG_1$ & $DG_2$ & $DG_3$ & $CG_1$ & $CG_2$ & $CG_3$ & $BDM_1$ & $BDM_2$ & $BDM_3$ & $RT_1$ & $RT_2$ & $RT_3$\\
\hline
\hline
$DG_\ell$ & -0.5 	& -0.5 	& -0.5 	& -0.5 	& -0.5 	& -0.5 	& -0.5 	& -0.5 	& -0.5 	& -0.5 	& -0.5 	& -0.5 	& -0.5 \\
\hline
$CG_1$ & 0.98 	& 0.96 	& 0.96 	& 0.94 	& 1.79 	& 1.44 	& 0.95 	& 1.48 	& 1.03 	& 0.95 	& 1.01 	& 1.02 	& 0.95 \\
\hline
$CG_2$ & 0.0 	& 1.97 	& 1.93 	& 1.91 	& 1.45 	& 1.85 	& 1.87 	& 1.47 	& 1.88 	& 1.88 	& 1.02 	& 1.83 	& 1.88 \\
\hline
$CG_3$ & 0.02 	& 1.5 	& \textbf{3.01} 	& 2.97 	& 0.95 	& 1.95 	& \textbf{3.43} 	& 0.98 	& 1.93 	& \textbf{3.26} 	& 0.57 	& 1.51 	& \textbf{3.08} \\
\hline
\end{tabular}

\vspace{2em}
\begin{tabular}{c|ccccccccccccc}
\hline 
$CG_2$ & $DG_0$	& $DG_1$ & $DG_2$ & $DG_3$ & $CG_1$ & $CG_2$ & $CG_3$ & $BDM_1$ & $BDM_2$ & $BDM_3$ & $RT_1$ & $RT_2$ & $RT_3$\\
\hline
\hline
$DG_\ell$ & -0.5 	& -0.5 	& -0.5 	& -0.5 	& -0.5 	& -0.5 	& -0.5 	& -0.5 	& -0.5 	& -0.5 	& -0.5 	& -0.5 	& -0.5 \\
\hline
$CG_1$ & 0.98 	& 0.96 	& 0.96 	& 0.94 	& 1.79 	& 1.44 	& 0.95 	& 1.48 	& 1.03 	& 0.95 	& 1.01 	& 1.02 	& 0.95 \\
\hline
$CG_2$ & 0.0 	& 1.97 	& 1.93 	& 1.91 	& 1.45 	& 1.85 	& 1.87 	& 1.47 	& 1.88 	& 1.88 	& 1.02 	& 1.83 	& 1.88 \\
\hline
$CG_3$ & 0.02 	& 1.5 	& \textbf{3.01} 	& 2.97 	& 0.95 	& 1.95 	& \textbf{3.43} 	& 0.98 	& 1.93 	& \textbf{3.26} 	& 0.57 	& 1.51 	& \textbf{3.08} \\
\hline
\end{tabular}

\vspace{2em}
\caption{Convergence at order $3$ for $\mathcal E^{\mathcal X}$ obtained for different combinations of finite element families. The first cell (the upper-left one) of each table gives the type of {\it boundary finite element}: $DG_2$ or $CG_2$, columns correspond to the $q$-type variables, and rows to the $p$-type ones. The order in boldface are the optimal order given by Theorem~\ref{th:Pk-Pl-State} and Proposition~\ref{prop:other-choices}.}
\label{tab:EX3}
\end{table}

All the convergence rates are presented on Tables~\ref{tab:EX1}-\ref{tab:EX2}-\ref{tab:EX3} for the absolute error $\mathcal E^{\mathcal X}$. The associated figures show the convergence rates proven in Theorem~\ref{th:General-State} and Proposition~\ref{prop:other-choices} for the optimal choices of finite element families among all our tests.

In Theorem~\ref{th:General-State}, $H_p$ is assumed to be $H^1(\Omega)$-conforming. Looking at Tables~\ref{tab:EX1}-\ref{tab:EX2}-\ref{tab:EX3}, this seems indeed necessary to ensure convergence. In every test cases, the rate is negative when $H_p$ is given by discontinuous Galerkin finte elements, no matter the order (we gather lines for $DG_\ell$, $\ell=0, 1, 2, 3$ since they give the same results). It has to be noted that assumption~\eqref{H2} is not satisfied in these cases: Theorem~\ref{th:General-State} does not apply.


\FloatBarrier
\subsection{About the convergence rate of the Hamiltonian error \texorpdfstring{$\mathcal E^{\mathcal{\Ham}}$}{EH}}

So far, \emph{compatibility conditions} have not been taken into account between $\H_q$, $H_p$ and $H_\partial$. In other words, only conforming assumptions have been made, and optimal rates have then been deduced. However, pHs are strongly structured, and in particular, the de Rham cohomology should be respected to improve the efficiency of the PFEM. As a motivation, it is remarkable on Tables~\ref{tab:EH2}--\ref{tab:EH2-4}--\ref{tab:EH2-4-6} that $\mathcal E^{\Ham}$ does not converge at the same rate as $\mathcal E^{\mathcal X}$, as stated in Corollary~\ref{cor:General-Hamiltonian}, but at twice its order in a various number of cases. 


\begin{table}[ht!]
\centering
\vspace{2em}
\begin{tabular}{c|ccccccccccccc}
\hline 
$DG_0$ & $DG_0$	& $DG_1$ & $DG_2$ & $DG_3$ & $CG_1$ & $CG_2$ & $CG_3$ & $BDM_1$ & $BDM_2$ & $BDM_3$ & $RT_1$ & $RT_2$ & $RT_3$\\
\hline
\hline
$DG_\ell$ & -1.0 	& -1.0 	& -1.0 	& -1.0 	& -1.0 	& -1.0 	& -1.0 	& -1.0 	& -1.0 	& -1.0 	& -1.0 	& -1.0 	& -1.0 \\
\hline
$CG_1$ & 0.83 	& 0.8 	& 0.81 	& 0.8 	& 0.8 	& 0.81 	& 0.8 	& 0.81 	& 0.81 	& 0.8 	& 0.82 	& 0.81 	& 0.8 \\
\hline
$CG_2$ & -0.1 	& 0.81 	& 0.82 	& 0.81 	& 0.66 	& 0.82 	& 0.81 	& 0.81 	& 0.82 	& 0.81 	& 0.83 	& 0.82 	& 0.81 \\
\hline
$CG_3$ & -0.1 	& 0.8 	& 0.81 	& 0.81 	& 1.09 	& 0.99 	& 0.81 	& 0.82 	& 0.81 	& 0.81 	& 0.84 	& 0.81 	& 0.81 \\
\hline
\end{tabular}

\vspace{2em}
\caption{Hamiltonian convergence rates for $DG_0$ boundary finite elements. The possible optimal order given by Theorem~\ref{th:General-Hamiltonian} is $\theta^* = 2$, \ie twice the order reached for the state error convergence rate.}
\label{tab:EH2}
\end{table}
On Table~\ref{tab:EH2}, one can see that the Hamiltonian convergence rates when boundary functions are approximated by Discontinuous Galerkin finite elements of order 0, $DG_0$, never reach order $2$, \ie twice the optimal convergence rate of the state error according to Theorem~\ref{th:Pk-Pl-State}. This leads us to conclude that compatibility conditions of Theorem~\ref{th:General-Hamiltonian} are never met for the combinations of finite elements presented on Table~\ref{tab:EH2}.

On the contrary, increasing by one order the approximation for boundary terms, both for Discontinuous and Continuous Galerkin $DG_1$ and $CG_1$, leads to an order $2$ in most cases, as seen on Table~\ref{tab:EH2-4}. Those corresponding to a convergence rate of order $1$ for the state error in Table~\ref{tab:EX2}  might indeed satisfy the compatibility condition of Theorem~\ref{th:General-Hamiltonian}.
\begin{table}[ht!]
\centering
\vspace{2em}
\begin{tabular}{c|ccccccccccccc}
\hline 
$DG_1$ & $DG_0$	& $DG_1$ & $DG_2$ & $DG_3$ & $CG_1$ & $CG_2$ & $CG_3$ & $BDM_1$ & $BDM_2$ & $BDM_3$ & $RT_1$ & $RT_2$ & $RT_3$\\
\hline
\hline
$DG_\ell$ & -1.0 	& -1.0 	& -1.0 	& -1.0 	& -1.0 	& -1.0 	& -1.0 	& -1.0 	& -1.0 	& -1.0 	& -1.0 	& -1.0 	& -1.0 \\
\hline
$CG_1$ & \textbf{2.05} 	& \textbf{1.99} 	& \textbf{2.07} 	& \textbf{2.07} 	& 1.96 	& 2.08 	& \textbf{2.07} 	& 2.04 	& \textbf{2.08} 	& \textbf{2.07} 	& \textbf{2.09} 	& \textbf{2.08} 	& \textbf{2.07} \\
\hline
$CG_2$ & 0.01 	& 1.95 	& 2.02 	& 2.02 	& 1.96 	& 2.02 	& 2.02 	& 2.06 	& 2.02 	& 2.02 	& \textbf{2.04} 	& 2.02 	& 2.02 \\
\hline
$CG_3$ & 0.04 	& 1.99 	& 2.02 	& 2.02 	& \textbf{2.11} 	& 2.02 	& 2.03 	& \textbf{2.0} 	& 2.02 	& 2.02 	& 1.14 	& 2.02 	& 2.03 \\
\hline
\end{tabular}

\vspace{2em}
\begin{tabular}{c|ccccccccccccc}
\hline 
$CG_1$ & $DG_0$	& $DG_1$ & $DG_2$ & $DG_3$ & $CG_1$ & $CG_2$ & $CG_3$ & $BDM_1$ & $BDM_2$ & $BDM_3$ & $RT_1$ & $RT_2$ & $RT_3$\\
\hline
\hline
$DG_\ell$ & -1.0 	& -1.0 	& -1.0 	& -1.0 	& -1.0 	& -1.0 	& -1.0 	& -1.0 	& -1.0 	& -1.0 	& -1.0 	& -1.0 	& -1.0 \\
\hline
$CG_1$ & \textbf{2.05} 	& \textbf{1.99} 	& \textbf{2.07} 	& \textbf{2.07} 	& 1.96 	& 2.08 	& \textbf{2.07} 	& 2.04 	& \textbf{2.08} 	& \textbf{2.07} 	& \textbf{2.09} 	& \textbf{2.08} 	& \textbf{2.07} \\
\hline
$CG_2$ & 0.01 	& 1.95 	& 2.02 	& 2.02 	& 1.96 	& 2.02 	& 2.02 	& 2.06 	& 2.02 	& 2.02 	& \textbf{2.04} 	& 2.02 	& 2.02 \\
\hline
$CG_3$ & 0.04 	& 1.99 	& 2.02 	& 2.02 	& \textbf{2.11} 	& 2.02 	& 2.03 	& \textbf{2.0} 	& 2.02 	& 2.02 	& 1.14 	& 2.02 	& 2.03 \\
\hline
\end{tabular}

\vspace{2em}
\caption{Hamiltonian convergence rates for $DG1$ and $CG1$ boundary finite elements. The possible optimal order given by Theorem~\ref{th:General-Hamiltonian} is $\theta^* = 2$ or $4$, \ie twice the orders reached for the state error convergence rate. In boldface the convergence rate achieving order $2$ while state error is only of order $1$, \ie a numerical evidence that compatibility conditions must be satisfied in those cases.}
\label{tab:EH2-4}
\end{table}

Analogously, increasing again the approximation at the boundary, \ie taking $DG_2$ or $CG_2$ finite elements, allows for a Hamiltonian convergence rate reaching order 4, inviting us to conjecture that compatibility conditions of Theorem~\ref{th:General-Hamiltonian} are met for boldface convergence rate of Table~\ref{tab:EH2-4-6}. Remark that we do not have a sufficiently small time step to be able to numerically observe order 6, if any. 
\begin{table}[ht!]
\centering
\vspace{2em}
\begin{tabular}{c|ccccccccccccc}
\hline 
$DG_2$ & $DG_0$	& $DG_1$ & $DG_2$ & $DG_3$ & $CG_1$ & $CG_2$ & $CG_3$ & $BDM_1$ & $BDM_2$ & $BDM_3$ & $RT_1$ & $RT_2$ & $RT_3$\\
\hline
\hline
$DG_\ell$ & -1.0 	& -1.0 	& -1.0 	& -1.0 	& -1.0 	& -1.0 	& -1.0 	& -1.0 	& -1.0 	& -1.0 	& -1.0 	& -1.0 	& -1.0 \\
\hline
$CG_1$ & \textbf{1.99} 	& \textbf{1.95} 	& \textbf{1.93} 	& \textbf{1.93} 	& 1.95 	& 1.93 	& \textbf{1.93} 	& 1.96 	& \textbf{1.91} 	& \textbf{1.93} 	& \textbf{2.0} 	& \textbf{1.91} 	& \textbf{1.93} \\
\hline
$CG_2$ & 0.01 	& 1.95 	& 1.88 	& \textbf{3.94} 	& 1.95 	& 1.98 	& \textbf{3.92} 	& 1.95 	& 1.77 	& \textbf{3.91} 	& \textbf{2.04} 	& 1.94 	& \textbf{3.93} \\
\hline
$CG_3$ & 0.04 	& 1.95 	& --- 	& 4.12 	& \textbf{1.95} 	& \textbf{4.04} 	& --- 	& \textbf{1.95} 	& \textbf{4.11} 	& --- 	& 1.14 	& 1.94 	& --- \\
\hline
\end{tabular}

\vspace{2em}
\begin{tabular}{c|ccccccccccccc}
\hline 
$CG_2$ & $DG_0$	& $DG_1$ & $DG_2$ & $DG_3$ & $CG_1$ & $CG_2$ & $CG_3$ & $BDM_1$ & $BDM_2$ & $BDM_3$ & $RT_1$ & $RT_2$ & $RT_3$\\
\hline
\hline
$DG_\ell$ & -1.0 	& -1.0 	& -1.0 	& -1.0 	& -1.0 	& -1.0 	& -1.0 	& -1.0 	& -1.0 	& -1.0 	& -1.0 	& -1.0 	& -1.0 \\
\hline
$CG_1$ & \textbf{1.99} 	& \textbf{1.95} 	& \textbf{1.93} 	& \textbf{1.93} 	& 1.95 	& 1.93 	& \textbf{1.93} 	& 1.96 	& \textbf{1.91} 	& \textbf{1.93} 	& \textbf{2.0} 	& \textbf{1.91} 	& \textbf{1.93} \\
\hline
$CG_2$ & 0.01 	& 1.95 	& 1.88 	& \textbf{3.94} 	& 1.95 	& 1.98 	& \textbf{3.92} 	& 1.95 	& 1.77 	& \textbf{3.91} 	& \textbf{2.04} 	& 1.94 	& \textbf{3.93} \\
\hline
$CG_3$ & 0.04 	& 1.95 	& --- 	& 4.12 	& \textbf{1.95} 	& \textbf{4.04} 	& --- 	& \textbf{1.95} 	& \textbf{4.11} 	& --- 	& 1.14 	& 1.94 	& --- \\
\hline
\end{tabular}

\vspace{2em}
\caption{Hamiltonian convergence rates for $DG2$ and $CG2$ boundary finite elements. The possible optimal order given by Theorem~\ref{th:General-Hamiltonian} is $\theta^* = 2, 4$ or $6$, \ie twice the orders reached for the state error convergence rate. In boldface the convergence rate achieving order $2$ (resp. $4$) while state error is only of order $1$ (resp. $2$), \ie a numerical evidence that compatibility conditions must be satisfied in those cases. Order $6$ can not be reached because of the time scheme.}
\label{tab:EH2-4-6}
\end{table}

To conclude, as pHs deal with a Hamiltonian functional, which can be seen as the primary object, the \emph{structure-preserving discretization} should mean that $\Ham$ has to be accurately discretized for both the power balance (PFEM) {\em and} the value of $\Ham$ in $\R$ (compatibility conditions). Together, the PFEM and the compatibility conditions seem to achieve this, making use of the Finite Element Method only. They give the maximal precision with the minimal number of degrees of freedom. Furthermore, the number of degrees of freedom at the boundary being very low in comparison to those of $q$- and $p$-type, it clearly appears that the choice of boundary finite elements proves crucial to reach the expected convergence rates with respect to the finite elements chosen within the domain $\Omega$ (as an example, compare the rate obtained with the discretization $CG_2 \times CG_3$ for $\Alphaq$ and $\alphap$ on Table~\ref{tab:EH2-4} with the same on Table~\ref{tab:EH2-4-6}).

\FloatBarrier
\subsection{More test cases}
\label{Sec:More-test-cases}

\subsubsection{A non-convex case: the \texorpdfstring{$L$}{L}-shaped domain}

In this section, we test a non-convex case, with an analytical solution given as in the previous tests (boundary control is again given by the restriction of the analytical solution to $\partial\Omega$), with a time step $dt=10^{-3}$. One can appreciate on Figure~\ref{fig:L-shape} how our results remain valid even for this more complicated domain.
\begin{figure}[ht!]
\centering
\includegraphics[width=0.495\textwidth]{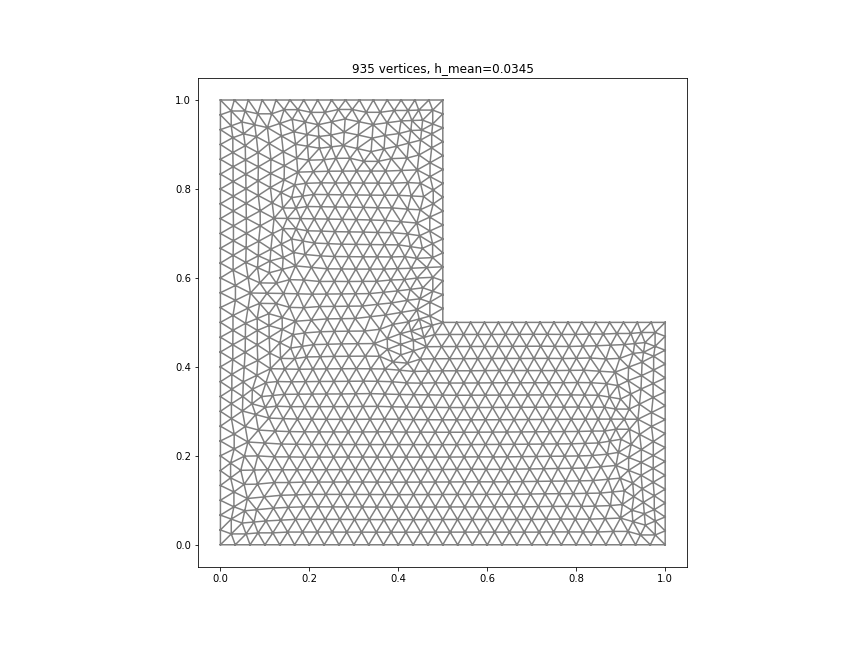}
\hfill
\includegraphics[width=0.495\textwidth]{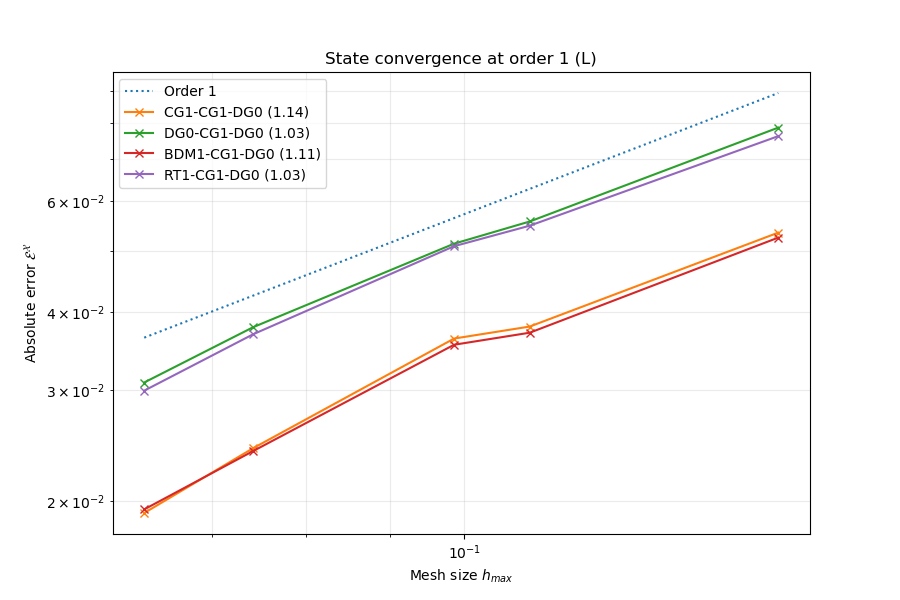} \\
\includegraphics[width=0.495\textwidth]{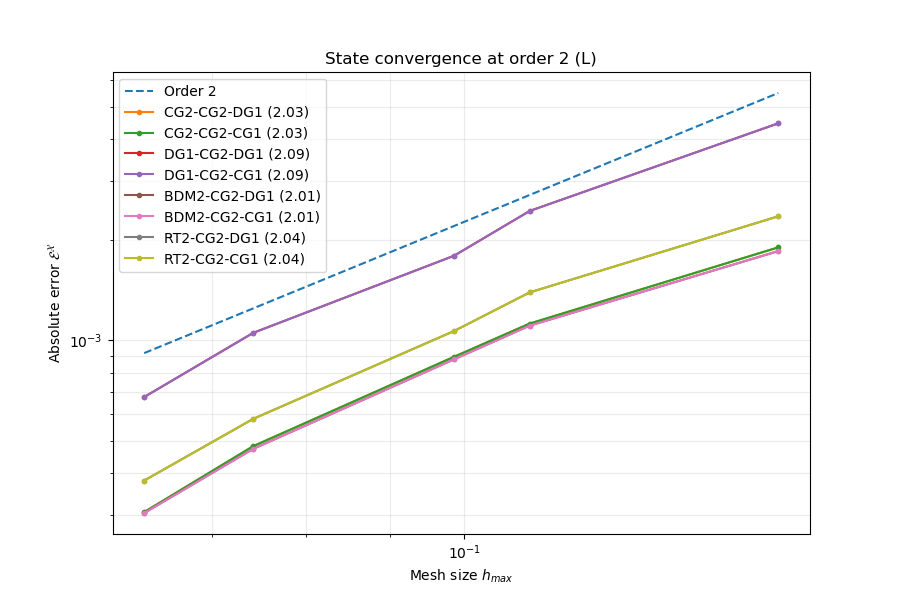}
\hfill
\includegraphics[width=0.495\textwidth]{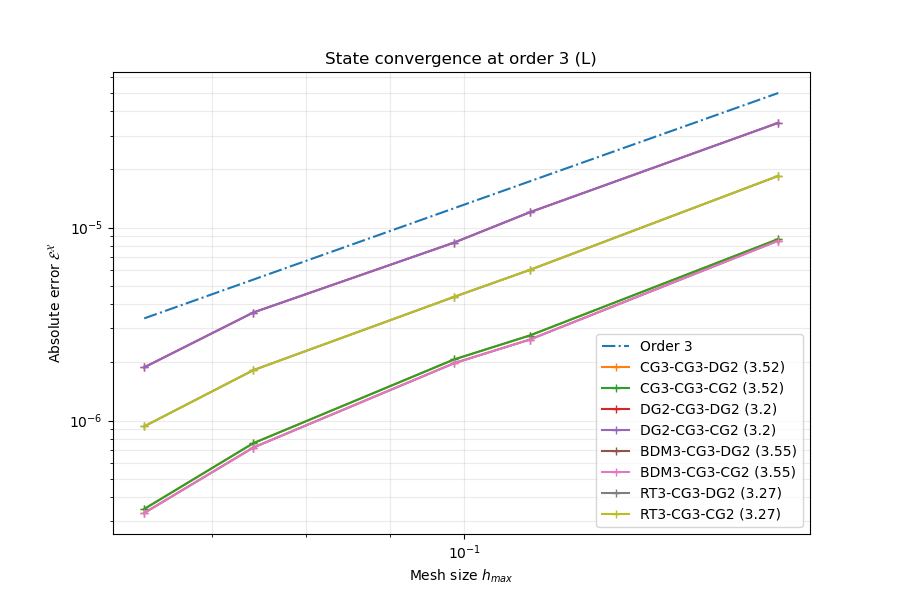}
\caption{$L$-shape: a non-convex domain.}\label{fig:L-shape}
\end{figure}

\FloatBarrier
\subsubsection{An anisotropic case}

As Theorems~\ref{th:General-State} and~\ref{th:General-Hamiltonian} are given for general heterogeneous anisotropic wave equations, a test case for constant anisotropy on the square $\Omega = (0,1) \times (0,1)$ is consider, with a time step $dt = 10^{-3}$. Let
$$
\Tens \equiv \matl 5 & 2 \\ 2 & 3 \matr, \qquad \rho \equiv 1,
$$
An analytical solution is then given by: $w(t,x) = \cos(3t-x+2y)$. More precisely
$$
\Alphaq = \matl -1 \\ 2 \matr \sin(3t-x+2y), \qquad \alpha_p = 3\sin(3t-x+2y).
$$
Convergence rates for optimal combinations of finite elements are given on Figure~\ref{fig:An}.

\begin{figure}[ht!]
\centering
\includegraphics[width=0.495\textwidth]{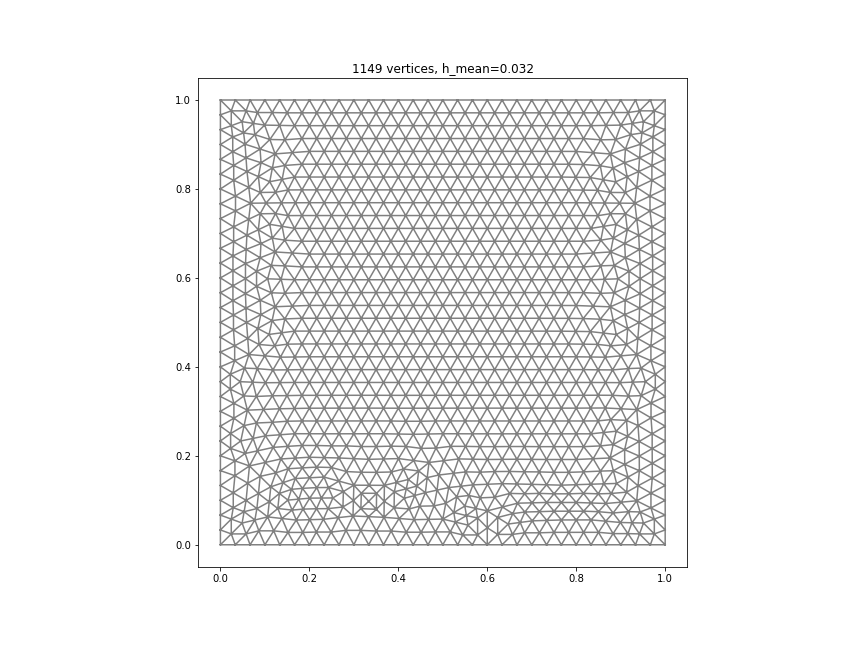}
\hfill
\includegraphics[width=0.495\textwidth]{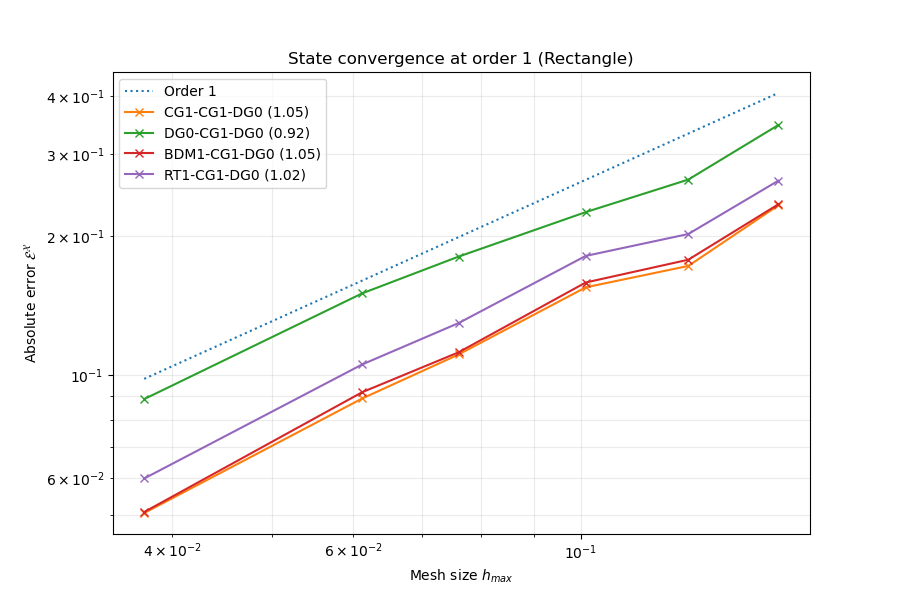} \\
\includegraphics[width=0.495\textwidth]{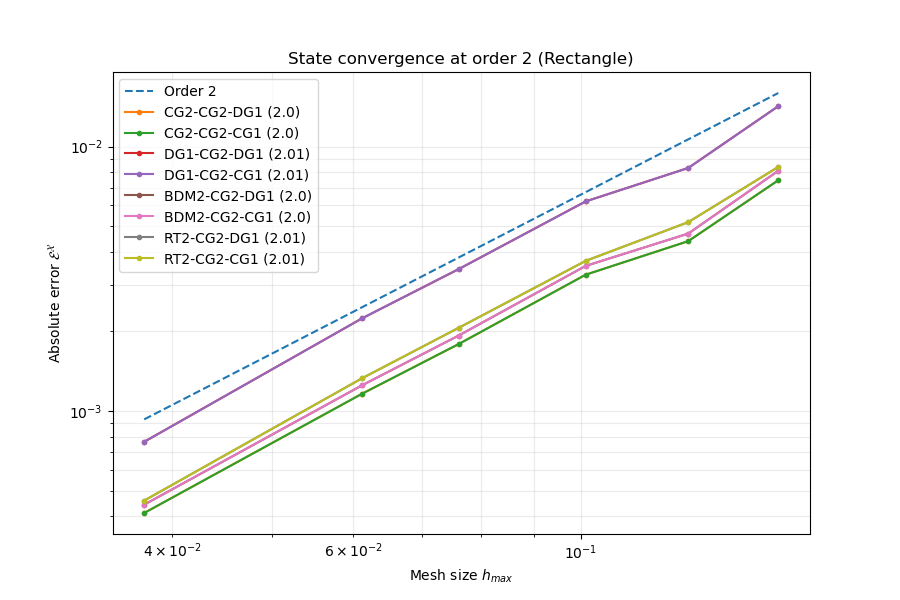}
\hfill
\includegraphics[width=0.495\textwidth]{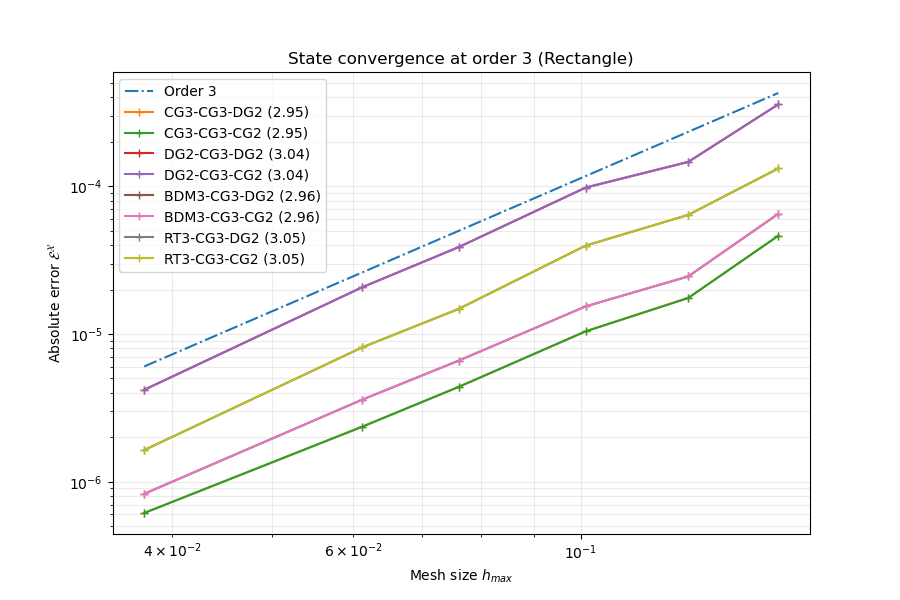}
\caption{An anisotropic case.}\label{fig:An}
\end{figure}

\FloatBarrier
\subsection{Absorbing boundary condition}
\label{Sec:Simu-damping}

It can be difficult to deal with dissipative boundary conditions 
for PDEs. 
For instance in the port-Hamiltonian formalism, at the continuous level, the dissipation does not appear as a positive bounded operator $\mathcal R$ in the dynamics in general. This means that the dynamics of the port-Hamiltonian systems is not necessarily governed by an operator of the form $(\mathcal J - \mathcal R) \mathcal Q$ (even with a lot of work involving lifting operators to fit this formulation, $\mathcal R$ will \emph{not} be a bounded operator). In other words, the dissipativity is \emph{hidden} in the domain of the unbounded operator $\mathcal J$. However, it is expected to recover a finite-dimensional pHs driven by a matrix of the form $\left( \mathcal J^d - \mathcal R^d \right) \mathcal Q^d$ at the discrete level.

PFEM provides a very easy and convenient way to describe this positive matrix $\mathcal R^d$ when dealing with admittance or impedance boundary conditions, \ie absorbing boundary condition in the PDE terminology. The strategy relies on the use of a suitable output feedback law on the \emph{finite-dimensional} pHs obtained by PFEM in the previous sections (see \cite{SerMatHai19a,SerMatHai19d} where this original strategy has been first proposed and explained in full details). Furthermore, this construction gives a well-understood structure to the matrix $\mathcal R^d$, which turns out to be of low rank: at most the dimension of $V_\partial$ as expected, since the damping is applied at the boundary only.

\subsubsection{Discretization}
More precisely, the following boundary condition is considered, instead of~\eqref{eq:waves-boundary}, for the admittance boundary condition
\begin{equation}\label{eq:waves-boundary-Y}\tag{\ref{eq:waves-boundary}Y}
\left\{\begin{array}{ll}
v(t,\x) = Y(\x) \partial_t w(t,\x) + \left( \Tens(\x) \; \grad (w(t,\x)) \right)^\top \; \n(\x),& \Forall \x \in \partial\Omega, t \ge 0,\\
y(t,\x) = \partial_t w(t,\x),& \Forall \x \in \partial\Omega, t \ge 0.
\end{array}\right.
\end{equation}
and the following one instead of~\eqref{eq:waves-boundary}, for the impedance boundary condition
\begin{equation}\label{eq:waves-boundary-Z}\tag{\ref{eq:waves-boundary}Z}
\left\{\begin{array}{ll}
\widetilde v(t,\x) = \partial_t w(t,\x) + Z(\x) \left( \Tens(\x) \; \grad (w(t,\x)) \right)^\top \; \n(\x),& \Forall \x \in \partial\Omega, t \ge 0,\\
\widetilde y(t,\x) = \left( \Tens(\x) \; \grad (w(t,\x)) \right)^\top \; \n(\x),& \Forall \x \in \partial\Omega, t \ge 0.
\end{array}\right.
\end{equation}
where both the admittance $Y$ and the impedance $Z$ are positive and belong to $L^\infty(\partial\Omega)$ and $v$ or $\widetilde v$ are the external inputs.

\begin{remark}
It is clear that~\eqref{eq:waves-boundary-Y} and~\eqref{eq:waves-boundary-Z} generalize~\eqref{eq:waves-boundary} and~\eqref{eq:waves-boundary-switch} respectively. Nevertheless, as mentioned above, PFEM does not apply straightforwardly in these more general cases (think about the use of Green's formula~\eqref{eq:Green} at the beginning of the strategy). The proposed alternative to construct the finite-dimensional dissipative pHs using an output feedback laws seems an elegant way to achieve our goal.
\end{remark}
It is easy to write the following relations: between $u$, $y$ and $v$ (the new control) from~\eqref{eq:waves-boundary} and~\eqref{eq:waves-boundary-Y}
$$
u(t,\x) = v(t,\x) - Y(\x) y(t,\x), \Forall \x \in \partial\Omega, t \ge 0,
$$
or between $\widetilde u$, $\widetilde y$ and $\widetilde v$ (the new control) from~\eqref{eq:waves-boundary-switch} and~\eqref{eq:waves-boundary-Z}
\begin{equation}\label{eq:feedback-switch}
\widetilde u(t,\x) = \widetilde v(t,\x) - Z(\x) \widetilde y(t,\x), \Forall \x \in \partial\Omega, t \ge 0.
\end{equation}
Using a weak formulation, these equalities read in matrix form
$$
M_\partial \underline{u} (t) = M_\partial \underline{v}(t) - \psl Y \psr \underline{y}(t), \Forall t \ge 0,
$$
or
$$
M_\partial \underline{\widetilde u} (t) = M_\partial \underline{\widetilde v}(t) - \psl Z \psr \underline{\widetilde y}(t), \Forall t \ge 0,
$$
respectively, where
$$
\psl Y \psr \eqdef \int_{\partial\Omega} Y(\s) \Psi(\s) \; (\Psi(\s))^\top \dd \s,
$$
or
$$
\psl Z \psr \eqdef \int_{\partial\Omega} Z(\s) \Psi(\s) \; (\Psi(\s))^\top \dd \s.
$$
As presented in~\cite[Remark 2.]{SerMatHai19d}, this procedure indeed gives rise to finite-dimensional Dirac structures, introducing extra \emph{resistive} ports, and leading to a pHDAE.

\subsubsection{Simulation results}

As a worked-out example, let us consider a fully heterogeneous and anisotropic case, with boundary control and boundary damping. The aim is to illustrate how the structure-preserving scheme can be appreciate on the Hamiltonian behaviour and the different kind of energies present in the system (potential, kinetic, supplied and damped).

Let us consider each part of the energy and their sum. The preservation of the physical meaning supposes that the exchanges of energy (\eg potential to kinetic and \textit{vice-versa}, boundary-supplied/taken energy to the system and damped into internal energy) must result in the preservation of the first principle of thermodynamics: the sum of all energies must be constant over time. More precisely, let us define the potential energy
$$
E_{\rm Pot}(t) \eqdef \frac{1}{2} \int_\Omega \left( \Alphaq(t,\x) \right)^\top \; \Tens(x,y) \; \Alphaq(t,\x) \; \dd \x,
$$
the kinetic energy
$$
E_{\rm Kin}(t) \eqdef \frac{1}{2} \int_\Omega \frac{\left( \alphaq(t,\x) \right)^2}{\rho(\x)} \; \dd \x,
$$
the boundary-supplied energy
$$
S(t) \eqdef \int_0^t \; \psl u(t,\s), y(t,\s) \psr_{\mathcal U, \mathcal Y} \; \dd t,
$$
and the damped energy
$$
D(t) \eqdef \int_0^t \; \psl Y(t,\s) \;y(t,\s), y(t,\s) \psr_{\mathcal U, \mathcal Y} \; \dd t.
$$
The total energy present in the system over time is then decomposed as
$$
E(t) = E_{\rm Pot}(t) + E_{\rm Kin}(t) + S(t) + D(t) = \Ham(t) + S(t) + D(t),
$$
and the first principle of thermodynamics implies that $E(t) \equiv E(0)$ for all $t\ge0$.

For our example on $\Omega \eqdef \left\{ \x \in \R^2 \; \mid \nol \x \nor < 1 \right\}$, the non-uniform anisotropic elasticity tensor $\Tens$ and the non-uniform heterogeneous mass density $\rho$ are taken as follows
$$
\Tens(x,y) \eqdef \matl 2 & 0.2(1+x)(1-x) \\ 0.2 (1+x) (1-x) & 1 \matr, 
\qquad 
\rho(x,y) \eqdef 2 + 0.25 (1+x) (1-x), \Forall (x,y) \in \Omega.
$$
The boundary control is chosen as
$$
u(t,x,y) \eqdef 
\left\lbrace\begin{array}{ll}
5 x \sin(t) \sin(1-t), & \Forall t < 1, (x,y) \in \partial\Omega, \\
0, & \Forall t > 1, (x,y) \in \partial\Omega,
\end{array}\right.
$$
while the admittance is defined by
$$
Y(t,x,y) \eqdef 
\left\lbrace\begin{array}{ll}
2.5 x \sin(t) \sin\left(\frac{t-1.5}{1.5}\right), & \Forall t > 1.5, (x,y) \in \partial\Omega, \\
0, & \Forall t < 1.5, (x,y) \in \partial\Omega,
\end{array}\right.
$$

The simulation is performed on the time interval $(0,3)$ with a time step $dt = 10^{-4}$. The spatial discretization is $RT_1 \times CG_1 \times CG_1$. The time solver is Assimulo (IDA SUNDIALS)~\cite{Andersson2015}. The integration in time to compute $S$ and $D$ are done using the midpoint rule. We can appreciate on Figure~\ref{fig:HetAnDamp} how the total energy $E$ remains constant over time, as physically expected, thanks to the PFEM.

\begin{figure}[ht!]
\centering
\includegraphics[width=0.8\textwidth]{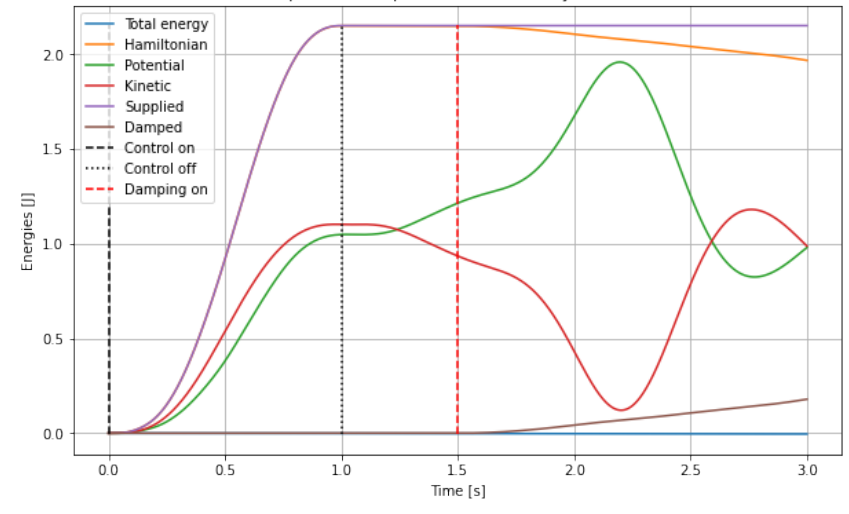}
\caption{Time evolution of the energies present in the system ($RT_1 \times CG_1 \times CG_1$ discretization). The number of degrees of freedom is: $2055+714+84=2853$.}\label{fig:HetAnDamp}
\end{figure}

\FloatBarrier
\section{Conclusion and perspectives}
\label{Sec:Conclusion}

In this work, the numerical analysis of the Partitioned Finite Element Method applied to the anisotropic and heterogeneous boundary-controlled-and-observed $N$-dimensional wave equation has been carried out. This recent structure-preserving method~\cite{CarMatLef18, CarMatLef19} allows the direct construction of a finite-dimensional port-Hamiltonian system, the underlying Dirac structure of which mimicks the infinite-dimensional Stokes-Dirac structure. This property allows a very accurate discretization of the power balance satisfied by the system. Furthermore, it has been shown that under compatibility conditions resembling those allowing for discrete de Rham complexes, the discrete Hamiltonian converges very fastly toward the continuous one, strengthening the interest of the PFEM, since the versatility of port-Hamiltonian systems aims precisely at modelling the exchange of energies between sub-systems. As an illustration of our main theorems, we have performed 2D simulations on a case where an analytical solution is known, with a non-homogeneous boundary condition: the boundary control, for both convex and non-convex domain. A wide range of usual finite element have been tested. Moreover, impedance-like boundary damping have been carried out to illustrate the structure-preserving property of the PFEM~\cite{SerMatHai19a,SerMatHai19d}.

Several questions remain open, the first one being the case of mixed boundary conditions. Two approaches have been proposed in~\cite{BruCarHaiKot20}: a domain decomposition followed by a gyrator interconnection between the two sub-systems, and the use of Lagrange multipliers. The former has the great advantage of remaining an ODE, while the latter transforms into a Differential Algebraic Equation (DAE). The numerical analysis of both alternatives will require deeper investigation.

A second interesting question is the problem of symplectic integration of pHDAE, which naturally arises in various situations such as the aforementioned mixed boundary conditions, or for the heat equation~\cite{SerMatHai19b,SerMatHai19c}. A scheme has been recently proposed in~\cite{MorMeh19} for this purpose.

Finally, in assumption~\eqref{H0} allowing for regular solutions, $\mathcal X_\kappa$ is assumed to be a Hilbert space. It is clear that this requires some regularity assumptions on the physical parameters. 
An interesting future work would be to design a structure-preserving scheme when the physical parameters $\rho$ and $\Tens$ are no more regular, \eg using mollifying methods: refining the constants in the estimates as function of those parameters, as done in this work, seems to be a first step in order to tackle this problem.


\bibliographystyle{siam} 
\bibliography{IJNAM}

\begin{thebibliography}{10}

\bibitem{PETScTS}
{\sc S.~Abhyankar, J.~Brown, E.~M. Constantinescu, D.~Ghosh, B.~F. Smith, and
  H.~Zhang}, {\em {PETSc/TS}: {A} modern scalable {ODE/DAE} solver library},
  tech. rep., 2018.

\bibitem{FEniCS}
{\sc M.~S. Aln\ae{}s, J.~Blechta, J.~Hake, A.~Johansson, B.~Kehlet, A.~Logg,
  C.~Richardson, J.~Ring, M.~E. Rognes, and G.~N. Wells}, {\em {The FEniCS
  Project Version 1.5}}, Archive of Numerical Software, 3 (2015).

\bibitem{altmann2020porthamiltonian}
{\sc R.~Altmann, V.~Mehrmann, and B.~Unger}, {\em {Port-Hamiltonian
  formulations of poroelastic network models}}, Mathematical and Computer
  Modelling of Dynamical Systems, 27 (2021), pp.~429--452.

\bibitem{Andersson2015}
{\sc C.~Andersson, C.~F\"uhrer, and J.~{\AA}kesson}, {\em Assimulo: A unified
  framework for {ODE} solvers}, Mathematics and Computers in Simulation, 116
  (2015), pp.~26--43.

\bibitem{Arnold2006}
{\sc D.~N. Arnold, R.~S. Falk, and R.~Winther}, {\em {Finite element exterior
  calculus, homological techniques, and applications}}, Acta Numerica, 15
  (2006), pp.~1--155.

\bibitem{ArnFalWin10}
\leavevmode\vrule height 2pt depth -1.6pt width 23pt, {\em {Finite element
  exterior calculus: from Hodge theory to numerical stability}}, American
  Mathematical Society. Bulletin. New Series, 47 (2010), pp.~281--354.

\bibitem{BecJolTso00}
{\sc E.~B\'{e}cache, P.~Joly, and C.~Tsogka}, {\em {An Analysis of New Mixed
  Finite Elements for the Approximation of Wave Propagation Problems}}, SIAM
  Journal on Numerical Analysis, 37 (2000), pp.~1053--1084.

\bibitem{BecJolTso02}
\leavevmode\vrule height 2pt depth -1.6pt width 23pt, {\em {A New Family of
  Mixed Finite Elements for the Linear Elastodynamic Problem}}, SIAM Journal on
  Numerical Analysis, 39 (2002), pp.~2109--2132.

\bibitem{BocHym06}
{\sc P.~B. Bochev and J.~M. Hyman}, {\em {Principles of Mimetic Discretizations
  of Differential Operators}}, in Compatible Spatial Discretizations, vol.~142
  of The IMA Volumes in Mathematics and its Applications, Springer, New York,
  2006, pp.~89--119.

\bibitem{BofBreFor13}
{\sc D.~Boffi, F.~Brezzi, and M.~Fortin}, {\em {Mixed Finite Element Methods
  and Applications}}, vol.~44 of Springer Series in Computational Mathematics,
  Springer, Berlin Heidelberg, 2013.

\bibitem{BruCarHaiKot20}
{\sc A.~Brugnoli, F.~L. Cardoso-Ribeiro, G.~Haine, and P.~Kotyczka}, {\em
  {Partitioned Finite Element Method for Power-Preserving Structured
  Discretization with Mixed Boundary Conditions}}, in 21st IFAC World Congress,
  vol.~53, Berlin, Germany, July 2020, IFAC, pp.~7647--7652.
\newblock (invited session).

\bibitem{CarMatLef18}
{\sc F.~L. Cardoso-Ribeiro, D.~Matignon, and L.~Lef{\`{e}}vre}, {\em {A
  structure-preserving Partitioned Finite Element Method for the 2D wave
  equation}}, IFAC-PapersOnLine, 51 (2018), pp.~119--124.
\newblock 6th IFAC Workshop on Lagrangian and Hamiltonian Methods for Nonlinear
  Control LHMNC 2018.

\bibitem{CarMatLef19}
{\sc F.~L. Cardoso-Ribeiro, D.~Matignon, and L.~Lef{\`{e}}vre}, {\em {A
  partitioned finite element method for power-preserving discretization of open
  systems of conservation laws}}, IMA Journal of Mathematical Control and
  Information, 38 (2021), pp.~493--533.

\bibitem{CarMatPom20}
{\sc F.~L. Cardoso-Ribeiro, D.~Matignon, and V.~Pommier-Budinger}, {\em
  {Port-Hamiltonian model of two-dimensional shallow water equations in moving
  containers}}, IMA Journal of Mathematical Control and Information, 37 (2020),
  pp.~1348--1366.

\bibitem{CerSchBan07}
{\sc J.~Cervera, A.~J. van~der Schaft, and A.~Ba\~{n}os}, {\em {Interconnection
  of port-Hamiltonian systems and composition of Dirac structures}},
  Automatica, 43 (2007), pp.~212--225.

\bibitem{Ces96}
{\sc M.~Cessenat}, {\em Mathematical methods in electromagnetism}, vol.~41 of
  Series on Advances in Mathematics for Applied Sciences, World Scientific
  Publishing Co. Inc., River Edge, NJ, 1996.
\newblock Linear theory and applications.

\bibitem{DuiMacStrBru09}
{\sc V.~Duindam, A.~Macchelli, S.~Stramigioli, and H.~Bruyninckx}, {\em
  {Modeling and Control of Complex Physical Systems: The Port-Hamiltonian
  Approach}}, Springer-Verlag, Berlin Heidelberg, 2009.

\bibitem{Egger2019}
{\sc H.~Egger}, {\em Structure preserving approximation of dissipative
  evolution problems}, Numerische Mathematik, 143 (2019), pp.~85--106.

\bibitem{EggKugLilMarMeh18}
{\sc H.~Egger, T.~Kugler, B.~Liljegren-Sailer, N.~Marheineke, and V.~Mehrmann},
  {\em {On Structure-Preserving Model Reduction for Damped Wave Propagation in
  Transport Networks}}, SIAM Journal on Scientific Computing, 40 (2018),
  pp.~A331--A365.

\bibitem{FarBalDyc14}
{\sc O.~Farle, R.-B. Baltes, and R.~Dyczij-Edlinger}, {\em {A Port-Hamiltonian
  Finite-Element Formulation for the Transmission Line}}, in 21st International
  Symposium on Mathematical Theory of Networks and Systems (MTNS), Groningen,
  The Netherlands, July 2014, MTNS, pp.~724--728.

\bibitem{FarKliJocFloDyc13}
{\sc O.~Farle, D.~Klis, M.~Jochum, O.~Floch, and R.~Dyczij-Edlinger}, {\em {A
  port-Hamiltonian finite-element formulation for the Maxwell equations}}, in
  2013 International Conference on Electromagnetics in Advanced Applications
  (ICEAA), Torino, Italy, September 2013, IEEE, pp.~324--327.

\bibitem{FenLiuWanZha21}
{\sc Y.~Feng, Y.~Liu, R.~Wang, and S.~Zhang}, {\em {A conforming discontinuous
  Galerkin finite element method on rectangular partitions}}, Electronic
  Research Archive, 29 (2021), pp.~2375--2389.

\bibitem{FoiTem78}
{\sc C.~Foia\c{s} and R.~M. Temam}, {\em {Remarques sur les {\'{e}}quations de
  Navier-Stokes stationnaires et les ph{\'{e}}nom{\`{e}}nes successifs de
  bifurcation}}, Annali della Scuola Normale Superiore di Pisa. Classe di
  Scienze. Serie IV, 5 (1978), pp.~28--63.

\bibitem{Gat14}
{\sc G.~N. Gatica}, {\em {A Simple Introduction to the Mixed Finite Element
  Method: Theory and Applications}}, SpringerBriefs in Mathematics, Springer,
  Cham, 2014.

\bibitem{Gerritsma2016}
{\sc M.~Gerritsma, J.~Kunnen, and B.~de~Heij}, {\em {Discrete Lie derivative}},
  in Lecture Notes in Computational Science and Engineering, vol.~112 of
  Lecture Notes in Computational Science and Engineering, Springer, Cham, 2016,
  pp.~635--643.

\bibitem{Gerritsma2018}
{\sc M.~Gerritsma, A.~Palha, V.~Jain, and Y.~Zhang}, {\em {Mimetic Spectral
  Element Method for Anisotropic Diffusion}}, in Numerical Methods for PDEs,
  vol.~15 of SEMA SIMAI Springer Series, Springer, Cham, 2018, pp.~31--74.

\bibitem{GirRav86}
{\sc V.~Girault and P.-A. Raviart}, {\em {Finite Element Methods for
  Navier-Stokes Equations}}, vol.~5 of Springer Series in Computational
  Mathematics, Springer, Berlin Heidelberg, 1986.

\bibitem{GugPolBeaSch12}
{\sc S.~Gugercin, R.~V. Polyuga, C.~Beattie, and A.~J. van~der Schaft}, {\em
  {Structure-preserving tangential interpolation for model reduction of
  port-Hamiltonian systems}}, Automatica, 48 (2012), pp.~1963--1974.

\bibitem{HaiMat21}
{\sc G.~Haine and D.~Matignon}, {\em {Structure-Preserving Discretization of a
  Coupled Heat-Wave System, as Interconnected Port-Hamiltonian Systems}}, in
  Geometric Science of Information, {Nielsen, Frank} and {Barbaresco,
  Fr{\'e}d{\'e}ric}, eds., vol.~12829 of Lecture Notes in Computer Science,
  Springer, Cham, 2021, pp.~191--199.

\bibitem{HaiMatMon2022}
{\sc G.~Haine, D.~Matignon, and F.~Monteghetti}, {\em Long-time behavior of a
  coupled heat-wave system using a structure-preserving finite element method},
  Mathematical Reports, 24(74) (2022), pp.~187--215.

\bibitem{Hiemstra2014}
{\sc R.~R. Hiemstra, D.~Toshniwal, R.~H.~M. Huijsmans, and M.~Gerritsma}, {\em
  {High order geometric methods with exact conservation properties}}, Journal
  of Computational Physics, 257 (2014), pp.~1444--1471.

\bibitem{Hiptmair2001b}
{\sc R.~Hiptmair}, {\em {Discrete Hodge operators}}, Numerische Mathematik, 90
  (2001), pp.~265--289.

\bibitem{Jol03}
{\sc P.~Joly}, {\em {Variational Methods for Time-Dependent Wave Propagation
  Problems}}, in Topics in Computational Wave Propagation: Direct and Inverse
  Problems, M.~Ainsworth, P.~Davies, D.~Duncan, B.~Rynne, and P.~Martin, eds.,
  vol.~31 of Lecture Notes in Computational Science and Engineering, Springer,
  Berlin, Heidelberg, 2003, pp.~201--264.

\bibitem{KotHDR}
{\sc P.~Kotyczka}, {\em Numerical Methods for Distributed Parameter
  Port-Hamiltonian Systems}, TUM.University Press, Munich, 2019.
\newblock Habilitation.

\bibitem{KotMasLef18}
{\sc P.~Kotyczka, B.~Maschke, and L.~Lef\`{e}vre}, {\em {Weak form of
  Stokes–Dirac structures and geometric discretization of port-Hamiltonian
  systems}}, Journal of Computational Physics, 361 (2018), pp.~442--476.

\bibitem{KurZwa15}
{\sc M.~Kurula and H.~Zwart}, {\em {Linear wave systems on n-{D} spatial
  domains}}, International Journal of Control, 88 (2015), pp.~1063--1077.

\bibitem{Kurula2010}
{\sc M.~Kurula, H.~Zwart, A.~J. van~der Schaft, and J.~Behrndt}, {\em {Dirac
  structures and their composition on Hilbert spaces}}, Journal of Mathematical
  Analysis and Applications, 372 (2010), pp.~402--422.

\bibitem{Lee2018b}
{\sc D.~Lee, A.~Palha, and M.~Gerritsma}, {\em {Discrete conservation
  properties for shallow water flows using mixed mimetic spectral elements}},
  Journal of Computational Physics, 357 (2018), pp.~282--304.

\bibitem{LepMorRod14}
{\sc F.~Lepe, D.~Mora, and R.~Rodr\'{i}guez}, {\em {Locking-free finite element
  method for a bending moment formulation of Timoshenko beams}}, Computers \&
  Mathematics with Applications, 68 (2014), pp.~118--131.

\bibitem{MorMeh19}
{\sc V.~Mehrmann and R.~Morandin}, {\em {Structure-Preserving Discretization
  for Port-Hamiltonian Descriptor Systems}}, in IEEE 58th Conference on
  Decision and Control (CDC), Nice, France, 2019, IEEE, pp.~6863--6868.

\bibitem{Mon03}
{\sc P.~Monk}, {\em {Finite element methods for Maxwell's equations}},
  Numerical Mathematics and Scientific Computation, Oxford University Press,
  New York, 2003.

\bibitem{OliPor16}
{\sc T.~Oliveira and A.~Portela}, {\em {Weak-form collocation – A local
  meshless method in linear elasticity}}, Engineering Analysis with Boundary
  Elements, 73 (2016), pp.~144--160.

\bibitem{Rashad2020}
{\sc R.~Rashad, F.~Califano, A.~J. van~der Schaft, and S.~Stramigioli}, {\em
  Twenty years of distributed port-{H}amiltonian systems: a literature review},
  IMA Journal of Mathematical Control and Information, 37 (2020),
  pp.~1400--1422.

\bibitem{Sayas2004}
{\sc F.-J. Sayas}, {\em Aubin-nitsche estimates are equivalent to compact
  embeddings}, BIT Numerical Mathematics, 44 (2004), pp.~287--290.

\bibitem{SerMatHai19b}
{\sc A.~Serhani, G.~Haine, and D.~Matignon}, {\em {Anisotropic heterogeneous
  $n$-D heat equation with boundary control and observation: I. Modeling as
  port-Hamiltonian system}}, IFAC-PapersOnLine, 52 (2019), pp.~51--56.
\newblock 3rd IFAC Workshop on Thermodynamic Foundations for a Mathematical
  Systems (TFMST).

\bibitem{SerMatHai19c}
\leavevmode\vrule height 2pt depth -1.6pt width 23pt, {\em {Anisotropic
  heterogeneous $n$-D heat equation with boundary control and observation: II.
  Structure-preserving discretization}}, IFAC-PapersOnLine, 52 (2019),
  pp.~57--62.
\newblock 3rd IFAC Workshop on Thermodynamic Foundations for a Mathematical
  Systems (TFMST).

\bibitem{SerMatHai19d}
{\sc A.~Serhani, D.~Matignon, and G.~Haine}, {\em {A Partitioned Finite Element
  Method for the Structure-Preserving Discretization of Damped
  Infinite-Dimensional Port-Hamiltonian Systems with Boundary Control}}, in
  Geometric Science of Information, {Nielsen, Frank} and {Barbaresco,
  Fr{\'e}d{\'e}ric}, eds., vol.~11712 of Lecture Notes in Computer Science,
  Springer, Cham, 2019, pp.~549--558.

\bibitem{SerMatHai19a}
\leavevmode\vrule height 2pt depth -1.6pt width 23pt, {\em {Partitioned Finite
  Element Method for port-Hamiltonian systems with Boundary Damping:
  Anisotropic Heterogeneous 2-D wave equations}}, IFAC-PapersOnLine, 52 (2019),
  pp.~96--101.
\newblock 3rd IFAC Workshop on Control of Systems Governed by Partial
  Differential Equations (CPDE).

\bibitem{SesSchSch14}
{\sc M.~Seslija, J.~M. Scherpen, and A.~J. van~der Schaft}, {\em {Explicit
  simplicial discretization of distributed-parameter port-Hamiltonian
  systems}}, Automatica, 50 (2014), pp.~369--377.

\bibitem{Toledo2020}
{\sc J.~Toledo, Y.~Wu, H.~Ramírez, and Y.~{Le Gorrec}}, {\em Observer-based
  boundary control of distributed port-{H}amiltonian systems}, Automatica, 120
  (2020), pp.~109--130.

\bibitem{TreRamGorKot18}
{\sc V.~Trenchant, H.~Ramirez, Y.~Le~Gorrec, and P.~Kotyczka}, {\em {Finite
  differences on staggered grids preserving the port-Hamiltonian structure with
  application to an acoustic duct}}, Journal of Computational Physics, 373
  (2018), pp.~673--697.

\bibitem{TucWei09}
{\sc M.~Tucsnak and G.~Weiss}, {\em Observation and control for operator
  semigroups}, Birkh\"{a}user Advanced Texts: Basler Lehrb\"{u}cher,
  Birkh\"{a}user Verlag, Basel, 2009.

\bibitem{SchJel14}
{\sc A.~van~der Schaft and D.~Jeltsema}, {\em {Port-Hamiltonian Systems Theory:
  An Introductory Overview}}, Foundations and
  Trends\textsuperscript{\textregistered} in Systems and Control, 1 (2014),
  pp.~173--378.

\bibitem{SchMas02}
{\sc A.~J. van~der Schaft and B.~Maschke}, {\em {Hamiltonian formulation of
  distributed-parameter systems with boundary energy flow}}, Journal of
  Geometry and Physics, 42 (2002), pp.~166--194.

\bibitem{VuLefMas16}
{\sc N.~M.~T. Vu, L.~Lef\`{e}vre, and B.~Maschke}, {\em {A structured control
  model for the thermo-magneto-hydrodynamics of plasmas in tokamaks}},
  Mathematical and Computer Modelling of Dynamical Systems, 22 (2016),
  pp.~181--206.

\end{thebibliography}

\appendix

\section{Technical lemmas}
\label{Sec:Appendix}

The following lemma gives an \emph{upper} bound for $\theta_{1,0}$ in~\eqref{H5}, namely $\theta_{0,1}$ from~\eqref{H4}, although not optimal for general geometry (\eg for convex domains).

\begin{lemma}\label{lem:inverse-inequality}
If~\eqref{H4} holds true, then there exists a constant $C_{0,1} > 0$ such that
$$
\nol P_p v_p - v_p \nor_{H^1(\Omega)} \le C_{0,1} ~ h^{-\theta_{1,0}} ~ \nol P_{1,p} v_p - v_p \nor_{H^1(\Omega)}, \Forall v_p \in H^1(\Omega),
$$
\ie such that~\eqref{H5} holds with $\theta_{0,1} = \theta_{1,0}$.
\end{lemma}

\begin{proof}
Let $v_p \in H^1(\Omega)$, one has
$$
\begin{array}{ll}
\nol P_p v_p - v_p \nor_{H^1(\Omega)} &\le \nol P_{1,p} v_p - v_p \nor_{H^1(\Omega)} + \nol P_p v_p - P_{1,p} v_p \nor_{H^1(\Omega)},\\
&= \nol P_{1,p} v_p - v_p \nor_{H^1(\Omega)} + \nol P_p \left( v_p - P_{1,p} v_p \right) \nor_{H^1(\Omega)},\\
&\le \nol P_{1,p} v_p - v_p \nor_{H^1(\Omega)} + \nol P_{1,p} v_p - v_p \nor_{L^2(\Omega)} + \nol \grad \left( P_p \left( v_p - P_{1,p} v_p \right) \right) \nor_{\L^2(\Omega)},\\
&\le 2 \nol P_{1,p} v_p - v_p \nor_{H^1(\Omega)} + C_{1,0} ~ h^{-\theta_{1,0}} ~ \nol P_p \left( v_p - P_{1,p} v_p \right) \nor_{L^2(\Omega)},\\
&\le 2 \nol P_{1,p} v_p - v_p \nor_{H^1(\Omega)} + C_{1,0} ~ h^{-\theta_{1,0}} ~ \nol P_{1,p} v_p - v_p \nor_{L^2(\Omega)},\\
&\le \left( 2 + C_{1,0} ~ h^{-\theta_{1,0}} \right) \nol P_{1,p} v_p - v_p \nor_{H^1(\Omega)},\\
&\le C_{0,1} ~ h^{-\theta_{1,0}} ~ \nol P_{1,p} v_p - v_p \nor_{H^1(\Omega)},
\end{array}
$$
where we have used $P_p P_{1,p} v_p = P_{1,p} v_p$ for all $v_p \in H^1(\Omega)$,~using~\eqref{H4}, $\nol P_p \nor_{\mathcal L(L^2(\Omega))} = 1$, and we have defined $C_{0,1} \eqdef 2 (h^*)^{\theta_{1,0}} + C_{1,0}$.
\end{proof}

We provide here the two technical lemmas used in the proof of Theorem~\ref{th:General-State}.
\begin{lemma}\label{lem:der-norm-square}
Under the assumptions of Theorem~\ref{th:General-State}, one has
\begin{multline*}
\frac{1}{2} \frac{\dd}{\dd t} \nol \matl \calPb_q & 0 \\ 0 & \mathcal P_p \matr \matl \Alphaq \\ \alphap \matr - \matl \Alphaq^d \\ \alphap^d \matr \nor_{\mathcal X}^2 
= \psl \grad \left( \rho^{-1} \left( \alphap - \mathcal P_p \alphap \right) \right), \Tens \; \left( \calPb_q \Alphaq - \Alphaq^d \right) \psr_{\L^2(\Omega)} \\
- \psl \Tens \; \left( \Alphaq - \calPb_q \Alphaq \right), \grad \left( \rho^{-1} \left( \mathcal P_p \alphap - \alphap^d \right) \right) \psr_{\L^2(\Omega)} 
+ \psl u - u^d, \gamma_0 \left( \rho^{-1} \left( \mathcal P_p \alphap - \alphap^d \right) \right) \psr_{\mathcal U, \mathcal Y}.
\end{multline*}
\end{lemma}

\begin{proof}
From the weak formulations~\eqref{eq:FV-continuous}--\eqref{eq:FV-discrete}, and thanks to the conformity of the finite element families, one has
$$
\psl \partial_t \Alphaq - \partial_t \Alphaq^d, \Tens \; \v_q^d \psr_{\L^2(\Omega)} = \psl \grad \left( \rho^{-1} \alphap - \frac{\alphap^d}{\rho} \right), \Tens \; \v_q^d \psr_{\L^2(\Omega)}, 
\forall \v_q^d \in \H_q = \Tens^{-1} \; \V_q \subset \Tens^{-1} \; \L^2(\Omega),
$$
and
\begin{multline*}
\psl \partial_t \alphap - \partial_t \alphap^d, \frac{v_p^d}{\rho} \psr_{L^2(\Omega)} = - \psl \Tens \; \Alphaq - \Tens \; \Alphaq^d, \grad \left( \frac{v_p^d}{\rho} \right) \psr_{\L^2(\Omega)} \\
+ \psl u - u^d, \gamma_0 \left( \frac{v_p^d}{\rho} \right) \psr_{\mathcal U, \mathcal Y}, \Forall v_p^d \in H_p = \rho V_p \subset \rho H^1(\Omega).
\end{multline*}
Summing the latter two equalities gives
\begin{multline*}
\psl \partial_t \Alphaq - \partial_t \Alphaq^d, \Tens \; \v_q^d \psr_{\L^2(\Omega)} + 
\psl \partial_t \alphap - \partial_t \alphap^d, \frac{v_p^d}{\rho} \psr_{L^2(\Omega)} 
= \psl \grad \left( \rho^{-1} \alphap - \frac{\alphap^d}{\rho} \right), \Tens \; \v_q^d \psr_{\L^2(\Omega)} \\
- \psl \Tens \; \Alphaq - \Tens \; \Alphaq^d, \grad \left( \frac{v_p^d}{\rho} \right) \psr_{\L^2(\Omega)} 
+ \psl u - u^d, \gamma_0 \left( \frac{v_p^d}{\rho} \right) \psr_{\mathcal U, \mathcal Y}, \Forall \v_q^d \in \H_q, \; v_p^d \in H_p.
\end{multline*}
Now by choosing $\v_q^d \eqdef \calPb_q \Alphaq - \Alphaq^d \in \H_q$ and $v_p^d \eqdef \mathcal P_p \alphap - \alphap^d \in H_p$, we get
\begin{multline*}
\psl \partial_t \Alphaq - \partial_t \Alphaq^d, \Tens \; \left( \calPb_q \Alphaq - \Alphaq^d \right) \psr_{\L^2(\Omega)} 
+ \psl \partial_t \alphap - \partial_t \alphap^d, \rho^{-1} \left( \mathcal P_p \alphap - \alphap^d \right) \psr_{L^2(\Omega)}  \\
= \psl \grad \left( \rho^{-1} \alphap - \frac{\alphap^d}{\rho} \right), \Tens \; \left( \calPb_q \Alphaq - \Alphaq^d \right) \psr_{\L^2(\Omega)}
- \psl \Tens \; \Alphaq - \Tens \; \Alphaq^d, \grad \left( \rho^{-1} \left( \mathcal P_p \alphap - \alphap^d \right) \right) \psr_{\L^2(\Omega)}  \\
+ \psl u - u^d, \gamma_0 \left( \rho^{-1} \left( \mathcal P_p \alphap - \alphap^d \right) \right) \psr_{\mathcal U, \mathcal Y}.
\end{multline*}
Thanks to the orthogonality of $\matl \calPb_q & 0 \\ 0 & \mathcal P_p \matr$ in $\mathcal X$, we have
\begin{multline*}
\psl \partial_t \Alphaq - \partial_t \Alphaq^d, \Tens \; \left( \calPb_q \Alphaq - \Alphaq^d \right) \psr_{\L^2(\Omega)} 
+ \psl \partial_t \alphap - \partial_t \alphap^d, \rho^{-1} \left( \mathcal P_p \alphap - \alphap^d \right) \psr_{L^2(\Omega)} \\ 
= \psl \partial_t \calPb_q \Alphaq - \partial_t \Alphaq^d, \Tens \; \left( \calPb_q \Alphaq - \Alphaq^d \right) \psr_{\L^2(\Omega)} 
+ \psl \partial_t \mathcal P_p \alphap - \partial_t \alphap^d, \rho^{-1} \left( \mathcal P_p \alphap - \alphap^d \right) \psr_{L^2(\Omega)},
\end{multline*}
leading to the announced result.
\end{proof}

\begin{lemma}\label{lem:third-term}
Under the assumptions of Theorem~\ref{th:General-State}, one has for all $h \in (0,h^*)$
$$
\nol \grad \left( \rho^{-1} \left( \alphap - \mathcal P_p \alphap \right) \right) \nor_{\L^2(\Omega)} 
\le \left( C_{0,1} C_{1,p} ~ h^{\theta_{1,p}-\theta_{0,1}} 
+ \frac{\rho^+ C_{1,0} C_p}{\sqrt{\rho_-}} ~ h^{\theta_p-\theta_{1,0}} \right) \nol \rho^{-1} \alphap \nor_{H^{\kappa+1}(\Omega)}.
$$
\end{lemma}

\begin{proof}
Writing
$$
\nol \grad \left( \rho^{-1} \left( \alphap - \mathcal P_p \alphap \right) \right) \nor_{\L^2(\Omega)} 
\le \nol \grad \left( \rho^{-1} \alphap - P_p \rho^{-1} \alphap \right) \nor_{\L^2(\Omega)}  
+ \nol \grad \left( P_p \rho^{-1} \alphap - \rho^{-1} \mathcal P_p \alphap \right) \nor_{\L^2(\Omega)},
$$
the first term on the right-hand side is bounded thanks to~\eqref{H5}, with $v_p = \rho^{-1} \alphap \in H^1(\Omega)$,
$$
\nol \grad \left( \rho^{-1} \left( \alphap - \mathcal P_p \alphap \right) \right) \nor_{\L^2(\Omega)} 
\le C_{0,1} ~ h^{-\theta_{0,1}} ~ \nol P_{1,p} \rho^{-1} \alphap - \rho^{-1} \alphap \nor_{H^1(\Omega)} 
+ \nol \grad \left( P_p \rho^{-1} \alphap - \rho^{-1} \mathcal P_p \alphap \right) \nor_{\L^2(\Omega)},
$$
and by~\eqref{H2}, still with $v_p = \rho^{-1} \alphap$,
$$
\nol \grad \left( \rho^{-1} \left( \alphap - \mathcal P_p \alphap \right) \right) \nor_{\L^2(\Omega)} 
\le C_{0,1} C_{1,p} ~ h^{\theta_{1,p}-\theta_{0,1}} ~ \nol \rho^{-1} \alphap \nor_{H^{\kappa+1}(\Omega)} 
+ \nol \grad \left( P_p \rho^{-1} \alphap - \rho^{-1} \mathcal P_p \alphap \right) \nor_{\L^2(\Omega)}.
$$
Since $(P_p \rho^{-1} \alphap - \rho^{-1} \mathcal P_p \alphap) \in V_p$, hypothesis~\eqref{H4} gives
$$
\nol \grad \left( \rho^{-1} \left( \alphap - \mathcal P_p \alphap \right) \right) \nor_{\L^2(\Omega)} 
\le C_{0,1} C_{1,p} ~ h^{\theta_{1,p}-\theta_{0,1}} ~ \nol \rho^{-1} \alphap \nor_{H^{\kappa+1}(\Omega)} 
+ C_{1,0} ~ h^{-\theta_{1,0}} ~ \nol P_p \rho^{-1} \alphap - \rho^{-1} \mathcal P_p \alphap \nor_{L^2(\Omega)}.
$$

Let us focus now on $\nol P_p \rho^{-1} \alphap - \rho^{-1} \mathcal P_p \alphap \nor_{L^2(\Omega)}$ to conclude. Since $\rho^{-1} \mathcal P_p \rho$ is a projector from $L^2(\Omega)$ onto $V_p$, one has $P_p \rho^{-1} \alphap = \rho^{-1} \mathcal P_p \rho \left( P_p \rho^{-1} \alphap \right)$. Hence
$$
\begin{array}{ll}
\dsp \nol P_p \rho^{-1} \alphap - \rho^{-1} \mathcal P_p \alphap \nor_{L^2(\Omega)} 
& \dsp = \nol \rho^{-1} \mathcal P_p \rho \left( P_p \rho^{-1} \alphap \right) - \rho^{-1} \mathcal P_p \alphap \nor_{L^2(\Omega)} \\
& \dsp \le \frac{1}{\sqrt{\rho_-}} \nol \frac{1}{\sqrt{\rho}} \mathcal P_p \left( \rho P_p \rho^{-1} \alphap - \alphap \right) \nor_{L^2(\Omega)} \\
& \dsp \le \frac{1}{\sqrt{\rho_-}} \nol \rho P_p \rho^{-1} \alphap - \rho \rho^{-1} \alphap \nor_{L^2(\Omega)} \\
& \dsp \le \frac{\rho^+}{\sqrt{\rho_-}} \nol P_p \rho^{-1} \alphap - \rho^{-1} \alphap \nor_{L^2(\Omega)},
\end{array}
$$
where we have used the lower bound $\rho_-$ for $\rho$ from the first to the second line. From the second to the third line, we have used the norm of the projector $\mathcal P_p$, which is 1 thanks to its orthogonality in $L^2(\Omega)$ endowed with the inner product $\psl v_1, \rho^{-1} v_2 \psr_{L^2}$, for all $v_1, v_2 \in L^2(\Omega)$. Finally, we have used the upper bound $\rho^+$ for $\rho$ from the third to the fourth line.

By~\eqref{H1}, still with $v_p = \rho^{-1} \alphap$, we get
$$
\nol P_p \rho^{-1} \alphap - \rho^{-1} \mathcal P_p \alphap \nor_{L^2(\Omega)} \le \frac{\rho^+ C_p}{\sqrt{\rho_-}} ~ h^{\theta_p} ~ \nol \rho^{-1} \alphap \nor_{H^{\kappa+1}(\Omega)},
$$
leading to the announced result.
\end{proof}

\end{document}